\documentclass[a4paper,11pt,reqno]{amsart}
\usepackage{amsmath}
\usepackage{relsize}
\usepackage[noadjust]{cite}
\usepackage{color}
\usepackage{graphicx}
\usepackage[dvipsnames]{xcolor}

\usepackage{hyperref, enumitem}
\hypersetup{
  colorlinks   = true, 
  urlcolor     = blue, 
  linkcolor    = Purple, 
  citecolor   = red 
}
\usepackage{amsmath,amsthm,amssymb}
\usepackage{setspace}
\usepackage{enumitem}
\setitemize{itemsep=-1pt}
\setenumerate{itemsep=-1pt}
\usepackage[margin=2.3cm]{geometry}
\usepackage{comment}
\usepackage{stmaryrd}
\usepackage{longtable}

\usepackage{pgf,tikz}
\usetikzlibrary{arrows}

\usepackage{array}
\newcolumntype{x}[1]{>{\centering\arraybackslash\hspace{0pt}}p{#1}}

\theoremstyle{definition}
\newtheorem{theorem}{Theorem}[section]
\newtheorem{definition}[theorem]{{{Definition}}}
\newtheorem{example}[theorem]{{{Example}}}
\newtheorem{notation}[theorem]{{{Notation}}}
\newtheorem{remark}[theorem]{{{Remark}}}

\newtheorem{corollary}[theorem]{{{Corollary}}}
\newtheorem{proposition}[theorem]{{{Proposition}}}
\newtheorem{lemma}[theorem]{{{Lemma}}}

\newcommand{\numberset}{\mathbb}
\newcommand{\At}{\mathrm{At}}
\newcommand{\N}{\numberset{N}}

\newcommand{\Z}{\numberset{Z}}

\newcommand{\F}{\numberset{F}}

\newcommand{\C}{\mathcal{C}}

\newcommand{\M}{\mathcal{M}}

\newcommand{\mU}{\mathcal{U}}

\newcommand{\kL}{\mathcal{L}}

\newcommand{\mZ}{\mathcal{Z}}

\newcommand{\mH}{\mathcal{H}}

\newcommand{\wt}{\textnormal{wt}}

\newcommand{\Fq}{\F_q}
\newcommand{\Fm}{\F_{q^m}}

\newcommand{\mL}{\mathcal{L}}

\newcommand{\mI}{\mathcal{I}}

\newcommand{\mS}{\mathcal{S}}

\newcommand{\hh}{\textnormal{dim}}

\newcommand{\rk}{\textnormal{rk}}

\DeclareMathOperator{\zero}{\mathbf{0}}
\DeclareMathOperator{\one}{\mathbf{1}}
\DeclareMathOperator{\GL}{GL}

\DeclareMathOperator{\PG}{PG}
\DeclareMathOperator{\cl}{cl}
\DeclareMathOperator{\cyc}{cyc}
\DeclareMathOperator{\Cl}{Cl}

\DeclareRobustCommand{\qq}{\emph{q}}

\allowdisplaybreaks

\makeatletter
\def\namedlabel#1#2{\begingroup
    #2%
    \def\@currentlabel{#2}%
    \phantomsection\label{#1}\endgroup
}
\makeatother

\title{The free product of \qq-matroids}

\usepackage[foot]{amsaddr}
\author{Gianira N. Alfarano}
\author{Eimear Byrne}
\author{Andrew Fulcher}

\address{University College Dublin, Ireland.}
\email{gianira.alfarano@gmail.com, ebyrne@ucd.ie, andrew.fulcher@ucdconnect.ie}

\begin{document}

\begin{abstract}
    We introduce the notion of the free product of $q$-matroids, which is the $q$-analogue of the free product of matroids.  We study the properties of this noncommutative binary operation, making an extensive use of the theory of cyclic flats. We show that the free product of two $q$-matroids $M_1$ and $M_2$ is maximal with respect to the weak order on $q$-matroids having $M_1$ as a restriction and $M_2$ as the complementary contraction. We characterise $q$-matroids that are irreducible with respect to the free product and  we prove that the factorization of a $q$-matroid into a free product of irreducibles is unique up to isomorphism. We discuss the representability of the free product, with a particular focus on rank one uniform $q$-matroids and show that such a product is represented by clubs on the projective line.
\end{abstract}

\maketitle


\section{Introduction}

The notion of a $q$-matroid was introduced by Crapo \cite{crapo1964theory}, however, since their rediscovery in 2018 in \cite{jurrius2018defining}, the study of $q$-matroids has attracted a great deal of attention due to the link to rank-metric codes; see for instance \cite{glues2021qpolyrankmetcode, gorla2019rank, ghorpade2020polymatroid, shiromoto2019codes}.  
Recently, the direct sum of $q$-matroids has been introduced in terms of the rank function; see~\cite{ceria2024direct}. The study of this operation highlighted many differences between matroid theory and $q$-matroid theory. 

The free product of matroids was introduced by Crapo and Schmitt in \cite{crapo2005free}. The same authors used it as a tool to prove Welsh’s 1969 conjecture~\cite{welsh1969bound}, which gives a lower bound on the number of isomorphism classes of matroids defined on a ground set of cardinality $n$. Moreover, in \cite{crapo2005unique}, it was shown that the free product of matroids $M$ and $N$ on ground sets $S$ and $T$ is the unique matroid with the most independent sets among all matroids on $S \cup T$ whose restriction to~$S$ and whose complementary contraction by $S$ are $M$ and $N$, respectively. They also showed that the direct sum of matroids has the most dependent sets in this class of matroids. In other words, the direct sum and the free product are the minimal and maximal element, respectively, with respect to the weak order among matroids whose restrictions to $S$ and complementary contraction by $S$ are $M$ and $N$, respectively.

As for matroids, there are many equivalent ways to describe a $q$-matroid axiomatically. An exposition of such \emph{cryptomorphisms} has been given in~\cite{byrne2022constructions}, in terms of the rank function, independent spaces, flats, circuits, bases, spanning spaces, the closure function, hyperplanes, open spaces etc. Another cryptomorphism based on cyclic flats was shown in \cite{alfarano2022cyclic}. These are, in general, not straightforward $q$-analogues of the traditional matroid cryptomorphisms.

In this paper we introduce the free product of a pair of $q$-matroids. We give three cryptomorphic descriptions of the free product, namely in terms of its independent spaces, its rank function, and its lattice of cyclic flats. 
We show that while the free product is the maximal element with respect to the weak order among a particular class of $q$-matroids, the direct sum is \emph{not} in general the minimal element, in contrast to the matroid case.
We then establish that any $q$-matroid can be uniquely factorised (up to isomorphism) into irreducible components with respect to the free product. We use the theory of cyclic flats of $q$-matroids, recently developed in \cite{alfarano2022cyclic, gluesing2023decompositions} as a key tool for many of our results.  

We also study the representability of the free product in terms of finite geometry, again making use of the lattice of cyclic flats. Geometric descriptions of representable $q$-matroids in terms of $q$-systems (see \cite{sheekey2019scatterd, randrianarisoa2020geometric})  have been given in \cite{alfarano2022cyclic}.
We describe representations of the free product (if one exists) and then we focus on representations of the free product of uniform $q$-matroids.
These are special $q$-matroids, which are representable as \emph{maximum rank distance (MRD) codes} or, equivalently, as \emph{scattered subspaces}. Inspired by \cite{alfarano2024representability}, using properties of cyclic flats, we prove that if a $q$-system is a representation of the free product of two uniform $q$-matroids, then it is \emph{evasive}.
Moreover we show the converse in the case of the free product of rank one uniform $q$-matroids, by relating it to the existence of \emph{clubs}.

This remainder of this paper is organized as follows.
Section \ref{sec:preliminaries} contains the required preliminary material. In section \ref{sec:cryptomorphisms}, we define the free product of two $q$-matroids in terms of the independent spaces of the factors. We then describe the lattice of cyclic flats of the free product, as well as an explicit expression of the rank function. 
In section \ref{sec:weak ordering}, we establish several fundamental properties of the free product; we give a duality result and establish associativity of the free product. 
We consider the weak order on a special subclass $\M_q(M_1,\dots,M_\ell)$ of $q$-matroids (see Definition \ref{def:poset}). We prove that the free
product of $q$-matroids is maximal in $\M_q(M_1,\dots,M_\ell)$ with respect to the weak order mentioned above. We consider decomposition properties of $q$-matroids with respect to the free product and show that reducibility is equivalent to the existence of free separators. We show, through a sequence of results, that any $q$-matroid can be factorised (up to isomorphism)
uniquely into irreducible components. As a corollary, we obtain a positive result on a $q$-analogue of Welsh's conjecture.  
In section \ref{sec:representability}, we investigate the representability of the free product. We show that if the free product is representable then it is so by a block upper triangular matrix whose diagonal blocks are representations of the factors. Then we focus on the representation of the free product of uniform $q$-matroids. We study its geometric description and we show that a $q$-system is a representation of the free product of rank one uniform $q$-matroids if and only if it gives rise to a club on the projective line. In section \ref{sec:conclusions}, we conclude with some open problems.

\section{Background}\label{sec:preliminaries}
In this section we establish the notation and provide the background material for the rest of the paper. 
We start by recalling some definitions concerning lattices. For a more detailed treatment, we refer the interested reader to \cite{Birkhoff}.

\begin{definition}
	Let $(\mathcal{L}, \leq)$ be a partially ordered set (poset). Let $a,b,v \in \mathcal{L}$.
	We say that $v$ is an \textbf{upper bound} (resp. \textbf{lower bound}) of $a$ and $b$ if $a\leq v$ and $b\leq v$ (resp. $v\leq a$ and $v\leq b$) and furthermore, we say	that $v$ is a \textbf{least upper bound}  (resp. \textbf{greatest lower bound}) of $a$ and $b$ if $v \leq u$ for any $u \in  \mathcal{L} $   that is also an upper bound of $a$ and $b$ (resp. $u\leq v$ for any $u\in\mL$ that is also a lower bound of $a$ and $b$).  
    If a least upper bound (resp. greatest lower bound) of $a$ and $b$ exists, then it is unique, is denoted by $a\vee b$ (resp. $a\wedge b$), which is called the \textbf{join} of $a$ and $b$ (resp. \textbf{meet}).
    The poset $\mathcal{L}$ is called a \textbf{lattice} if each pair of elements has a least upper bound and a greatest lower bound and it is denoted by $(\mathcal{L}, \leq, \vee, \wedge)$. 
    An element in $\mL$ that is not smaller than any other element is called \textbf{maximal} element of $\mL$ and it is denoted by $\boldsymbol{1}_{\mL}$ and an element that is not bigger than any other element is called \textbf{minimal} element of $\mL$ and it is denoted by $\boldsymbol{0}_{\mL}$. If there is no confusion, we simply write $\boldsymbol{1}$ and $\boldsymbol{0}$.
\end{definition}

\begin{definition}
Let $\mathcal{L}$ be a lattice with meet $\wedge$ and join $\vee$.
Let $a, b\in\mathcal{L}$ be such that $a \leq b$.
\begin{enumerate}
    \item An {\bf interval} $[a,b]\subseteq\mathcal{L}$ is the set of all $x\in\mathcal{L}$ such that $a\leq x\leq b$. It defines the {\bf interval sublattice} $([a,b],\leq,\vee,\wedge)$.
    \item If $[a,b]\subseteq \mL$ is such that for any $x\in \mL$, $x\in[a,b]$ implies that $x=a$ or $x=b$, then $b$ is called a {\bf cover} of $a$ and we write $a\lessdot b$.
    $\mH([a,b]):=\{x \in [a,b] : x \lessdot b\}$. We also define
    $\At(b):=\{x \in [\zero,b] : \zero \lessdot \;x \}$ and $\mH(b):=\{x \in [\zero,b] : x \lessdot b\}$.
    \item 
    A {\bf chain} from $a$ to $b$ is a totally ordered subset of $[a,b]$ with respect to $\leq$. A chain from $a$ to $b$ is called a {\bf maximal} chain in $[a,b]$ if it is not properly contained in any other chain from $a$ to $b$. 
    A finite {\bf chain} from $a$ to $b$ is a sequence of the form
    $ a = x_1 < \cdots < x_{k+1}=b $ with $x_j\in\mathcal{L}$ for $j \in  \{1,\dots,k\}$, in which case we say that the chain has length $k$.
\end{enumerate}
\end{definition}

Let $\F_q$ be the finite field with $q$ elements and let $E$ be an $n$-dimensional vector space over $\F_q$. In this paper, we are interested in the subspace lattice $(\mathcal{L}(E), \leq, \vee, \wedge)$, which is the lattice of $\mathbb{F}_q$-subspaces of $E$, ordered with respect to inclusion and for which the join is the usual vector space sum and the meet is the subspace intersection. That is, for all subspaces $A,B \in \mL(E)$ we have:
$ A \vee B = A + B, \; A \wedge B = A \cap B.$
The minimal element of $\mL(E)$ is $\zero=\langle 0 \rangle$ and its maximal element is $\textbf{1}=E$. For the sake of simplicity, we write $\zero$ to denote the minimal element of a subspace lattice of a vector space of any dimension.
For each $U \in \mL(E)$, we write $U^\perp$ to denote the orthogonal complement of $U$ with respect to a fixed non-degenerate bilinear form on $E$. The map $U\mapsto U^\perp$ is an involutory anti-automorphism of~$\mL(E)$.

\begin{definition}
    A \textbf{$q$-matroid with ground space} $E$ is a pair $M=(E,r)$, where $r$ is an integer-valued function defined on $\mL(E)$ with the following properties:
	\begin{enumerate}
		\item[\textbf{(R1)}] Boundedness: $0\leq r(A) \leq \hh(A)$, for all $A\in \mL(E)$. 
		\item[\textbf{(R2)}] Monotonicity: $A\leq B \Rightarrow r(A)\leq r(B)$, for all $A,B\in\mL(E)$. 
		\item[\textbf{(R3)}] Submodularity: $r(A+ B)+r(A\cap B)\leq r(A)+r(B)$, for all $A,B\in\mL(E)$.  
	\end{enumerate}
 The function $r$ is called \textbf{rank function} and the value $r(M) := r(E)$ is the \textbf{rank of the $q$-matroid}.
\end{definition}

\begin{definition}
    Let $E_1$ and $E_2$ be $\Fq$-vector spaces of finite dimension $n$. Two $q$-matroids $M_1=(E_1,r_1)$ and $M_2=(E_2,r_2)$ are called {\bf isomorphic} 
    if there exists an $\Fq$-isomorphism $\tau: E_1 \longrightarrow E_2$ such that $r_2(\tau(V))=r_1(V)$ for all $V \in \mL(E_1)$. We also say that $M_1$ and $M_2$ are {\bf equivalent} $q$-matroids.
    We say that $M_1$ and $M_2$ are {\bf lattice-equivalent} if there exists a lattice isomorphism
    $\xi: \mL(E_1) \longrightarrow \mL(E_2)$ such that $r_2(\xi(V))=r_1(V)$ for all $V \in \mL(E_1)$.
\end{definition}

Let $M=(E,r)$ be a $q$-matroid. A one-dimensional space $x\in\mL(E)$ is a \textbf{loop} if $r(x)=0$. A codimension 1 subspace of $E$ that has rank less than $r(E)$ is called a \textbf{coloop}. A subspace $A\in\mL(E)$ is \textbf{independent} if $r(A)=\dim(A)$ and \textbf{dependent} otherwise. The inclusion-maximal independent spaces are called \textbf{bases} and the inclusion-minimal dependent spaces are called \textbf{circuits}. A space $A\in\mL(E)$ is a \textbf{flat} if it is inclusion-maximal in the set $\{V \in \mL(E) : r(V)=r(A)\}$, and a space is \textbf{cyclic} or \textbf{open} if it is the sum of circuits. Finally, a subspace which is both a flat and cyclic is called a \textbf{cyclic flat}. 

For a given $q$-matroid $M=(E,r)$, we have the following two operators. 
For each $A \in \kL(E)$, define $$\Cl_r(A):=\{x \in \kL(E): r(A+x)=r(A)\}.$$
The {\bf closure operator} of a $q$-matroid $(E,r)$ is the function defined by
\[\cl_r:\mathcal{L}(E) \to\mathcal{L}(E), \  A \mapsto \cl_r(A):=\sum_{x \in \Cl_r(A)} x .
\]
 The {\bf cyclic operator} of $M$ is the function defined by
\begin{align*}
\cyc_r:\mathcal{L}(E) \to\mathcal{L}(E), \ A \mapsto \cyc_r(A):=\sum_{\substack{C \leq A \\C \textnormal{ is cyclic}}} C.
\end{align*} 

If $M=(E,r)$, we denote by $\mathcal{I}(M)$ and $\mZ(M)$ the collection of independent spaces and cyclic flats of $M$ and by $\cl_r$ and $\cyc_r$ its closure and cyclic operators. 
If it is clear from the context, we will simply write $\mI, \mZ, \cl,\cyc$.

It is well known that the collection of independent spaces uniquely determines the $q$-matroid, and the same is true for the collections of dependent spaces, open spaces, flats and circuits.

For our purposes, we recall the following version of the independence axioms.

\begin{notation}
Let $\mathcal{A}$ be a collection of subspaces of $E$. For any subspace $X \in \mL(E)$, we define then the collection of \textbf{maximal subspaces} of $X$ in
$\mathcal{A}$ to be the collection of subspaces
    $$\max(X, \mathcal{A}) := \{A \in \mathcal{A} : A \leq X \textnormal{ and } B <X, B \in\mathcal{A} \Rightarrow \dim(B) \leq \dim(A)\}.$$
\end{notation}

\begin{definition}[\cite{byrne2022constructions}]\label{def:indep_ax}
    Let $\mI$ be a collection of subspaces of $E$. We define the following {\bf independence axioms}.
    \begin{enumerate}
        \item[\namedlabel{i1}{{\rm (I1)}}] $\mI\ne\emptyset$.
        \item[\namedlabel{i2}{{\rm (I2)}}] For all $I, J \in \mL(E)$, if $J \in\mI$ and $I \leq J$, then $I \in \mI$.
        \item[\namedlabel{i3}{{\rm (I3)}}] Lat $I,J \in\mI$ be such that $\dim I < \dim J$. Then there exists $x\leq J$, $x\nleq I$, such that $I+x\in\mI$.
        \item[\namedlabel{i4}{{\rm (I4'')}}] Let $A \in \mL(E)$ and let $I \in \max(A, \mI)$. Let $ x\in \mL(E)$ be a one-dimensional space. Then there exists $J\in\max(A + x, \mI)$ such that $J \leq I + x$.
    \end{enumerate}
    If $\mI$ satisfies the independence axioms, we say that $\mI$  is a collection of \textbf{independent spaces} and we denote by $(E,\mI)$ the $q$-matroid defined by $\mI$.
\end{definition}

Furthermore, in \cite{alfarano2022cyclic} it is shown that the cyclic flats, together with their rank values, uniquely determine the $q$-matroid. We recall the following cyclic flat axioms. 

\begin{definition}{\cite[Definition 3.1]{alfarano2022cyclic}}\label{def:axioms_cyclic_flats}
    Let $\mZ$ be a collection of subspaces of $E$ and let
    $f:\mZ \rightarrow \Z$ be a map. 
    We define the following {\bf cyclic flat} axioms.
    \begin{enumerate}
    \item[\namedlabel{z0}{{\rm (Z0)}}] $(\mZ,\leq,\vee ,\wedge)$ is a lattice with join $\vee$ and meet $\wedge$, such that for every $Z_1,Z_2\in~\mZ$, we have that $Z_1+Z_2\leq Z_1\vee Z_2$ and $Z_1\wedge Z_2\leq Z_1\cap Z_2$, respectively. 
    \item[\namedlabel{z1}{{\rm (Z1)}}] We have that $f(\boldsymbol{0}_\mZ)=0$, where $\boldsymbol{0}_\mZ$ is the minimal element of $\mZ$.
    \item[\namedlabel{z2}{{\rm (Z2)}}] For every $F,G\in\mZ$ such that $G< F$, we have: 
    \begin{align*}\label{eq:differenceranks}
        0<f(F)-f(G)<\hh(F) -\hh(G).
    \end{align*}
    \item[\namedlabel{z3}{{\rm (Z3)}}] For every $F,G\in\mZ$ we have:
    \begin{equation*} \label{eq:submodularity2}f(F)+f(G)\geq f(F \vee G) + f(F\wedge G) + \hh((F\cap G)/(F\wedge G)). \end{equation*} 
\end{enumerate}
If $(\mZ,f)$ satisfies the cyclic flat axioms, we say that $\mZ$ is a lattice of {\bf cyclic flats} with respect to~$f$.
\end{definition}

The following result explains how to determine the rank function of the entire $q$-matroid from the cyclic flats together with their rank values; see \cite{alfarano2022cyclic, gluesing2023decompositions}.

\begin{lemma}{\cite[Corollary 3.12]{alfarano2022cyclic}}\label{lem:rank_function_cyclic_flats}
    Let $M=(E,r)$ be a $q$-matroid and $\mZ$ be its collection of cyclic flats. For all $A\in\mL(E),$ we have that
    $$r(A)=\min\{r(Z)+\hh((A+Z)/Z) : Z\in\mZ\}.$$
\end{lemma}

The independent spaces and the cyclic flats of a $q$-matroid are related by the following result.

\begin{lemma}\cite[Lemma 2.28]{alfarano2022cyclic}\label{lem:indep_charact}
    Let $M=(E,r)$ be a $q$-matroid. Then $I\in\mI(M)$ if and only if for every cyclic flat $Z\in\mZ(M)$, $\dim(I\cap Z)\leq r(Z)$. 
\end{lemma}

We recall the restriction and the contraction operations for $q$-matroids; see \cite{byrne2022constructions,jurrius2018defining}.

\begin{definition}\label{def:restriction}
 Let $M=(E,r)$ be a $q$-matroid and $A\in\mL(E)$.
 For every space $T\leq A$, we define $r_{M|{A}}(T):=r(T)$. The $q$-matroid  $M|{A}:=(A,r_{M|{A}})$ is called the \textbf{restriction of $M$ to $A$}. Define a map
 $$r_{M/A} : \mathcal{L}(E/A) \to \mathbb{N}_0, \ T\mapsto r(\pi^{-1}(T))-r(A),$$
where $\pi: E \to E/A$ is the canonical projection.
 Then the $q$-matroid $M/A :=(E/A,r_{M/A})$ is called the \textbf{contraction of $M$ by $A$}. For an arbitrary interval $[A,B]\subseteq\mL(E)$, we denote by $M[A,B]$, the $q$-matroid minor with rank function $r_{M[A,B]}$ defined on the interval $[A,B]$ by $r_{M[A,B]}(X)=r(X)-r(A)$ for $X\in[A,B]$.
\end{definition}

We recall the notion of {\em dual matroid} and some basic properties, which we will use in section~\ref{sec:weak ordering}.

\begin{definition}
Let $M=(E,r)$ be a $q$-matroid. The \textbf{dual $q$-matroid} of $M$ is the $q$-matroid $(M,r^\ast)$, where
\begin{align*}
    r^\ast: \mathcal{L}(E)\to \N_0, \ A \mapsto \dim(A)-r(E)+r(A^\perp).
\end{align*} 
\end{definition}

\begin{lemma}\cite[Lemma 11]{byrne2022weighteddesigns}\label{lem:duality_properties}
    Let $M = (E, r)$ be a $q$-matroid and let $T\in\mL(E)$. Then,
    $$M^\ast/T\cong (M|{T^\perp})^\ast, \qquad (M/T)^\ast\cong M^\ast|{T^\perp}.$$
\end{lemma}

As in the classical case, a well-known construction of a $q$-matroid arises from matrices; see \cite{jurrius2018defining}. Let $G$ be a $k \times n$ matrix with entries in the finite field $\F_{q^m}$ and for every $U\in\mL(\F_q^n)$, let $A^U$ be a matrix whose columns form a basis of $U$. Then the map 
\begin{equation*}
    r^G:\mL(\F_q^n) \to \mathbb{N}_0, \  U\mapsto \rk(G A^U),
\end{equation*}
is the rank function of a $q$-matroid, which we denote by $M[G]$ and we call the \textbf{$q$-matroid represented by $G$}. A $q$-matroid $M$ with ground space $\F_q^n$  and rank $k$ is \textbf{$\F_{q^m}$-representable} if $M=M[G]$ for some full rank $k\times n$ matrix $G$ with entries in $\F_{q^m}$. The $q$-matroid $M$ is called \textbf{representable} if it is \textbf{$\F_{q^m}$-representable} over the finite field $\F_{q^m}$.

\begin{example}
Let $\mU_{k,n}(q)$ be the \textbf{uniform $q$-matroid} of rank $k$ with ground space $\F_q^{n}$ that is, the $q$-matroid whose rank function is given by $r(A)=\min\{k,\dim(A)\}$ for all $A\in \mL(\F_q^n)$. Then $\mU_{0,n}(q)$ is represented by the $1\times n$ zero matrix, $\mU_{n,n}(q)$ is represented by the $n \times n$-identity matrix $\mathrm{Id}_n$, and, for $0 < k < n$, $\mU_{k,n}(q)$ is $\F_{q^m}$-representable if and only if $m \geq n$. This is because such a representation generates an MRD code, which exists if and only if $m\geq n$; see e.g.~\cite{delsarte1978bilinear}.
\end{example}

Finally, we recall a characterization of uniform $q$-matroids in terms of their cyclic flats.

\begin{lemma}\cite[Proposition 3.30]{alfarano2022cyclic}\label{lem:cyc_flats_uniform}
     Let $M=(E,r)$ be a $q$-matroid of rank $k$, with $0<k<n$. Let $\mZ$ be the lattice of cyclic flats of $M$. Then $M\cong\mU_{k,n}(q)$ if and only if $\mZ=\{\boldsymbol{0}_{\mZ}=\zero,\boldsymbol{1}_{\mZ}=E\}$ and $r(\boldsymbol{1}_{\mZ})=k$.   
\end{lemma}


\section{The free product of \qq-matroids}\label{sec:cryptomorphisms}
In this section, we introduce the free product of a pair of $q$-matroids. We provide cryptomorphic definitions of this product in terms of its independent spaces, its rank function and its lattice of cyclic flats. We note that while the independent spaces and the cyclic flats are similar to the matroid case, the rank function is somewhat different.

\subsection{Independent spaces of the free product}

\begin{notation}\label{not:section_3}
For any direct sum $E=E_1\oplus E_2$ of $\F_q$-vector spaces $E_1$ and $E_2$ we denote by
$\pi_i:E\longrightarrow E_i$  the corresponding projections and by $\iota_i:E_i\longrightarrow E$ the canonical embeddings. For $A\leq E_1$ and $B\leq E_2$ we may denote $\iota_1(A)\oplus\iota_2(B)$ by $A\oplus B$, when there is no risk of confusion. For instance, we may denote $A\leq E_1$ by $A\oplus \boldsymbol{0}\leq E$. For an arbitrary set $\mathcal{V}\subseteq \mL(E_2)$ and an arbitrary space $A\in\mL(E_1)$, we denote by $A\oplus\mathcal{V}$ the set $\{A \oplus V : V\in\mathcal{V}\}$. The notation $\mathcal{V}\oplus A$ is similarly defined. Analogously, we will denote by $\iota_1(\mathcal{V})$ (resp. $\iota_2(\mathcal{V})$) the set $\{\iota_1(V) : V\in\mathcal{V}\}$ (resp. $\{\iota_2(V) : V\in\mathcal{V}\}$).

For a $q$-matroid $M=(E,r)$ let $\lambda$ be the \textbf{rank-lack} function and $\nu$ be the \textbf{nullity} function, defined respectively by 
\begin{align}
    \lambda(A) &=r(E)-r(A),\\
    \nu(B) &= \hh(B)-r(B), 
\end{align}
for every $A,B\in\mL(E)$.

Let $M_1=(E_1,\mI_1)$ and $M_2=(E_2,\mI_2)$ be $q$-matroids. Let $r_1,\lambda_1,\nu_1$ and $r_2,\lambda_2,\nu_2$ be the rank,  rank-lack and nullity functions of $M_1$ and $M_2$, respectively.
\end{notation}

Define the set 
\begin{equation}\label{eq:independent}
    \mathcal{I}:=\{I\leq E_1\oplus E_2:\pi_1(I\cap\iota_1(E_1))\in\mathcal{I}_1,\lambda_1(\pi_1(I\cap\iota_1(E_1)))\geq\nu_2(\pi_2(I))\}.
\end{equation}

We now show that the set $\mI$ forms the collection of the independent spaces of a $q$-matroid. We will use the following results from linear algebra.

\begin{lemma}\label{lem:rank_nullity}
Let $A\in\mathcal{L}(E_1\oplus E_2)$. We have that $\textup{dim}(A)=\textup{dim}(A\cap\iota_1(E_1))+\textup{dim}(\pi_2(A))$.
\end{lemma}
\begin{proof}
For every $A\in\mL(E_1\oplus E_2)$, consider the linear map $\phi:A\rightarrow\pi_2(A)$, defined by $\phi(x)=\pi_2(x)$, for every $x\in A$. It is straightforward to check that $\textup{ker}(\phi)=A\cap \iota_1(E_1)$. By the rank-nullity theorem, we thus have that $\textup{dim}(A)=\textup{dim}(A\cap \iota_1(E_1))+\textup{dim}(\pi_2(A))$.
\end{proof}

\begin{lemma}\label{lem:lin_alg_1}
   Let $A,B\leq E_1\oplus E_2$ be such that $B\leq A$. Then the following hold.
   \begin{enumerate}
       \item $\pi_i(B\cap\iota_i(E_i))\leq\pi_i(A\cap\iota_i(E_i))$, for $i=1,2$.
       \item $\pi_i(B)\leq\pi_i(A)$, for $i=1,2$.
   \end{enumerate}
\end{lemma}
    
\begin{theorem}\label{thm:free_prod_indep_axiom}
The set $\mI$ defined in~\eqref{eq:independent} is the collection of independent spaces of a $q$-matroid with ground space $E_1\oplus E_2$.
\end{theorem}
\begin{proof}
Let $\mathcal{L}:=\mathcal{L}(E_1\oplus E_2)$. We will prove that $\mI$ satisfies the independence axioms from Definition~\ref{def:indep_ax}.
Clearly, $\boldsymbol{0}\in\mathcal{I}$, hence \ref{i1} is satisfied. Further, by Lemma~\ref{lem:lin_alg_1}, \ref{i2} is also satisfied. 

In order to show \ref{i3}, consider $I, J\in\mI$, with $\dim(I) < \dim(J)$. Then we have that 
\begin{align*}
\pi_1(I\cap\iota_1(E_1))\in\mI_1, \lambda_1(\pi_1(I\cap\iota_1(E_1)))&\geq \nu_2(\pi_2(I)),\\
    \pi_1(J\cap\iota_1(E_1))&\in\mI_1,\\
    \lambda_1(\pi_1(J\cap\iota_1(E_1)))&\geq \nu_2(\pi_2(J)).
\end{align*} 
We distinguish between three cases. 
\\
\textbf{Case 1:} Assume that $\lambda_1(\pi_1(I\cap\iota_1(E_1)))> \nu_2(\pi_2(I))$ and $\dim(\pi_1(I\cap\iota_1(E_1)))<\dim(\pi_1(J\cap\iota_1(E_1)))$. Then for every $x\leq \pi_1(J\cap \iota_1(E_1)), x \not \leq \pi_1(I\cap\iota_1(E_1))$, we have 
\begin{align*}
    \lambda_1(\pi_1(I\cap\iota_1(E_1))+x) &=r_1(E_1) -r_1(\pi_1(I\cap\iota_1(E_1))+x) \\
    &\geq r_1(E_1) -r_1(\pi_1(I\cap\iota_1(E_1)))-1 \\
    &=\lambda_1(\pi_1(I\cap\iota_1(E_1)))-1\geq \nu_2(\pi_2(I)).
\end{align*}
Choose $x\leq \pi_1(J\cap \iota_1(E_1)), x \not\leq\pi_1(I\cap\iota_1(E_1))$ such that $\pi_1(I\cap\iota_1(E_1))+x\in\mI_1$, which exists by the fact that $\mathcal{I}_1$ satisfies \ref{i3} by assumption. Finally, we let $x'=\iota_2(x)\leq J$ and conclude that $I+x'\in\mI$.
\\
\textbf{Case 2:}  Assume $\lambda_1(\pi_1(I\cap\iota_1(E_1)))> \nu_2(\pi_2(I))$ and $\dim(\pi_2(I))<\dim(\pi_2(J))$. Then for every $y\leq\pi_2(J)$, such that $y\not\leq\pi_2(I)$ we have that $\lambda_1(\pi_1(I\cap\iota_1(E_1)))\geq \nu_2(\pi_2(I))+1\geq\nu_2(\pi_2(I)+y)$. Hence, $I+\iota_2(y)\in\mI$. Since $y\leq\pi_2(J)$, there exists $y'\leq J$ such that $y=\pi_2(y')$. It follows that $I+y'\in\mathcal{I}$.
\\
\textbf{Case 3:} Assume $\lambda_1(\pi_1(I\cap\iota_1(E_1)))= \nu_2(\pi_2(I))$. First of all notice that $$\lambda_1(\pi_1(J\cap\iota_1(E_1)))=r_1(E_1)-\dim(\pi_1(J\cap\iota_1(E_1)))\geq \nu_2(\pi_2(J)).$$
Hence, we have that 
\begin{align*}
    r_2(\pi_2(I))&= \dim(\pi_2(I))-\nu_2(\pi_2(I)) \\
    &=\dim(\pi_2(I))- \lambda_1(\pi_1(I\cap\iota_1(E_1))) \\
    &=\dim(\pi_2(I)) -r_1(E_1) + \dim(\pi_1(I\cap\iota_1(E_1)))\\
    &< \dim(\pi_2(J)) -r_1(E_1) + \dim(\pi_1(J\cap\iota_1(E_1)))\\
    &\leq \dim(\pi_2(J)) - \nu_2(\pi_2(J)) =r_2(\pi_2(J)).
\end{align*}
Hence we can find a one-dimensional space $y\leq \pi_2(J), y \not\leq \pi_2(I)$ such that $\nu_2(\pi_2(I)+y) = \nu_2(\pi_2(I))$ and $I+\iota_2(y)\in\mI$.

In order to show \ref{i4}, let $A\in\mathcal{L}$, let $\beta\in\max(A,\mI)$ and let $x\in\textup{At}(\mathcal{L})$. Assume that $x\nleq A$.
Since $\beta\in\mI$, we have that $\pi_1(\beta\cap\iota_1(E_1))\in\mI(M)$ and $\lambda_M(\pi_1(\beta\cap\iota_1(E_1)))\geq\nu_N(\pi_2(\beta))$.
We distinguish between two cases.
\\
\textbf{Case 1:} Let $\lambda_1(\pi_1(\beta\cap\iota_1(E_1)))>\nu_2(\pi_2(\beta))$. Then we must have that $\pi_1(\beta\cap\iota_1(E_1))\in\max(A\cap\iota_1(E_1),\mI)$ and $\pi_2(\beta)=\pi_2(A)$, since otherwise there would exist a one-dimensional space $e<\pi_2(A)$ not in $\pi_2(\beta)$ such that $\beta<\beta+e\in\mathcal{I}$, which contradicts the maximality of $\beta$ in $A$.
\begin{itemize}[leftmargin=*]
    \item Let $x\leq\iota_1(E_1)$.
    \begin{itemize}
        \item If $\pi_1((\beta+x)\cap\iota_1(E_1))\in\mathcal{I}(M)$ we have that $\beta+x\in\mathcal{I}$. Furthermore, we have that $\pi_2(A+x)=\pi_2(A)=\pi_2(\beta)=\pi_2(\beta+x)$. Therefore, $\beta+x\in\max(A,\mI)$. This contradicts the maximality of $\beta$ in $A$.
        \item Suppose now that $\pi_1((\beta+x)\cap\iota_1(E_1))\notin\mathcal{I}_1$. Let $\alpha\leq A+x$ such that $\beta\lessdot\alpha$ and $\beta+x\neq\alpha$. If $\beta\cap\iota_1(E_1)<\alpha\cap\iota_1(E_1)$, then $\alpha\notin\mathcal{I}$. If $\beta\cap\iota_1(E_1)=\alpha\cap\iota_1(E_1)$, then by Lemma~\ref{lem:rank_nullity}, we have that $\pi_2(\beta)<\pi_2(\alpha)$. Therefore, $\pi_2(\alpha)=\pi_2(A+x)=\pi_2(\beta+x)$, which yields a contradiction since $x\leq A\cap\iota_1(E_1)$. Therefore, $\beta\in\max(A+x,\mI)$.
    \end{itemize}
    \item Let $x\nleq\iota_1(E_1)$.
    \begin{itemize}
        \item Let $\pi_1(\beta\cap\iota_1(E_1))<\pi_1((\beta+x)\cap\iota_1(E_1))$. We then have that $\beta+x=\beta+y$ for some $y\in\iota_1(E_1)$, which means we are in a case that we have already dealt with.
        \item Let $\pi_1(\beta\cap\iota_1(E_1))=\pi_1((\beta+x)\cap\iota_1(E_1))$. By Lemma~\ref{lem:rank_nullity} we have that $\pi_2(\beta)<\pi_2(\beta+x)$. Therefore, we have that $\pi_2(\beta+x)=\pi_2(A+x)$. By Lemma~\ref{lem:rank_nullity}, we have that $(A+x)\cap\iota_1(E_1)=A\cap\iota_1(E_1)$, which means that $\beta\cap\iota_1(E_1)\in\max((A+x)\cap\iota_1(E_1), \mI)$. This means that $\beta+x\in\max(A+x,\mI)$.
        \end{itemize}
    \end{itemize}
\textbf{Case 2:} Let $\lambda_1(\pi_1(\beta\cap\iota_1(E_1)))=\nu_2(\pi_2(\beta))$.
    \begin{itemize}[leftmargin=*]
        \item Let $x\leq\iota_1(E_1)$. By Lemma~\ref{lem:rank_nullity} we have that $\pi_2(\beta)=\pi_2(\beta+x)$ and $\pi_2(A)=\pi_2(A+x)$.
        \begin{itemize}
        \item If $\pi_1((\beta+x)\cap\iota_1(E_1))\in\mathcal{I}_1$ we have that $\lambda_1((\beta+x)\cap\iota_1(E_1))<\nu_2(\pi_2(\beta))=\nu_1(\pi_2(\beta+x))$ and thus $\beta+x\notin\mathcal{I}$. Let $\alpha\leq A+x$ such that $\beta\lessdot\alpha$ and $\alpha\neq\beta+x$. If $\beta\cap\iota_1(E_1)<\alpha\cap\iota_1(E_1)$, then similarly, $\alpha\notin\mathcal{I}$. If $\pi_2(\beta)<\pi_2(\alpha)$, then since $x\leq\iota_1(E_1)$ we have that $\pi_2(\alpha)\leq\pi_2(A)$, which means that $\lambda_1(\alpha\cap\iota_1(E_1))=\lambda_1(\beta\cap\iota_1(E_1))<\nu_2(\pi_2(\alpha))$. Therefore, $\alpha\notin\mathcal{I}$, which means that $\beta\in\max(A+x,\mI)$.
        \item $\pi_1((\beta+x)\cap\iota_1(E_1))\notin\mathcal{I}_1$ we have that $\beta+x\notin\mathcal{I}$. By similar reasoning as in the last point, we have that $\beta\in\max(A+x,\mI)$.
        \end{itemize}
    \item Let $x\nleq\iota_1(E_1)$.
    \begin{itemize}
        \item Let $\pi_2(\beta)<\pi_2(\beta+x)$. If $\nu_2(\pi_2(\beta))=\nu_2(\pi_2(\beta+x))$, then $\beta+x\in\mathcal{I}$. If there exists $\alpha\leq\pi_2(A+x)$ such that $\pi_2(\beta+x)\lessdot\alpha$ and $\nu_2(\pi_2(\beta+x))=\nu_2(\alpha)$, then there must exist $e\in\textup{At}(\pi_2(A))$ such that $\pi_2(\beta)\lessdot\pi_2(\beta+e)\neq\pi_2(\beta+x)$ and $\nu_2(\pi_2(\beta))=\nu_2(\pi_2(\beta+e))$ by the $q$-matroid properties of $M_2$. This contradicts the maximality of $\beta$, so therefore no such atom $e$ can exist. Hence, we have that $\beta+x\in\max(A+x,\mI)$.
        \item Let $\pi_2(\beta)=\pi_2(\beta+x)$. By Lemma~\ref{lem:rank_nullity} we have that $\beta\cap\iota_1(E_1)<(\beta+x)\cap\iota_1(E_1)$. Therefore, $\beta+x=\beta+y$ for some $y\leq\iota_1(E_1)$, which puts us in a case that we already dealt with. 
        \end{itemize}
    \end{itemize}
    In all the above cases we get that $\beta$ or $\beta+x$ belong to $\max(A+x,\mI)$, and this shows \ref{i4}. Hence, this concludes the proof.
\end{proof}

\begin{definition}
    For $q$-matroids $M_1=(E_1,\mI_1)$ and $M_2=(E_2,\mI_2)$ we call the $q$-matroid with ground space $E_1\oplus E_2$ and collection of independent spaces $\mI$ as defined in \eqref{eq:independent} the \textbf{free product of $M_1$ and $M_2$} and we denote it by $M_1\square M_2$.
\end{definition}

We will continue to use Notation~\ref{not:section_3} throughout sections \ref{sec:rank_funct_fp} and \ref{sec:cyc_flats_fp}.

\subsection{The rank function of the free product}\label{sec:rank_funct_fp}

In this subsection, we derive a compact expression for the rank function of the free product. 
Given the collection $\mI$ as in \eqref{eq:independent}, define the rank function of $M_1\square M_2$ as $r_\mI:\mathcal{L}(E_1\oplus E_2)\rightarrow\mathbb{N}_0,$ such that for every $X\in\mathcal{L}(E_1\oplus E_2)$, 
    $$r_{\mI}(X)=\textup{max}\{\dim(X\cap I):I\in\mathcal{I}\}.$$

    \begin{theorem}\label{thm:free_prod_rank_function}
        The function $r_\mI$ satisfies $$r_\mI(X)=r_1(\pi_1(X\cap\iota_1(E_1)))+r_2(\pi_2(X))+\textup{min}\{\lambda_1(\pi_1(X\cap\iota_1(E_1))),\nu_2(\pi_2(X))\},$$
        for every $X\in \mL(E_1\oplus E_2)$.
    \end{theorem}
    \begin{proof}
        Let $X\in\mL(E_1\oplus E_2)$. Since $M_1\square M_2$ is a $q$-matroid, we have that all bases of $X$ have equal dimension. Therefore we choose $\beta\in\max(X,\mI)$ such that $\beta\cap\iota_1(E_1)\in\max(X\cap\iota_1(E_1),\mI)$. By Lemma~\ref{lem:rank_nullity} we have the following.
        \begin{align}\label{eq:rank_cryptomorph}
            r_\mI(X)=\dim(\beta)&=\dim(\beta\cap\iota_1(E_1))+\dim(\pi_2(\beta))\nonumber\\
            &=r_1(X\cap\iota_1(E_1))+r_2(\pi_2(X))+\nu_2(\pi_2(\beta)).
        \end{align}
        Note that $\lambda_1(X\cap\iota_1(E_1))=\lambda_1(\beta\cap\iota_1(E_1))$ and that $\lambda_1(\beta\cap\iota_1(E_1))\geq\nu_2(\pi_2(\beta))$ since $\beta\in\mathcal{I}$. If $\lambda_1(\beta\cap\iota_1(E_1))=\nu_2(\pi_2(\beta))$, then we can substitute $\lambda_1(X\cap\iota_1(E_1))$ into (\ref{eq:rank_cryptomorph}) to get the result. If $\lambda_1(\beta\cap\iota_1(E_1))>\nu_2(\pi_2(\beta))$, then we must have that $\pi_2(\beta)=\pi_2(X)$, since otherwise there would exist an atom $e\in\textup{At}(\pi_2(X))\backslash\textup{At}(\pi_2(\beta))$ such that $\beta<\beta+e\in\mathcal{I}$, which contradicts the maximality of $\beta$ in $X$. The result follows.
    \end{proof}

From now on, we write $r$ in place of $r_\mI$, unless it is needed.

The following statement follows from a straightforward computation and hence we omit a proof.
\begin{lemma}\label{lem:ext_cases_rank}
    Let $X\in\mL(E_1\oplus E_2)$. The following holds.
    \begin{enumerate}
        \item If $X=\overline{X}\oplus \boldsymbol{0}$, for some $\overline{X}\in\mL(E_1)$, then $r(X)=r_1(\overline{X})$.
        \item If $X=E_1\oplus \overline{X}$, for some $\overline{X}\in\mL(E_2)$, then $r(X)=r_1(E_1)+r_2(\overline{X})$.
    \end{enumerate}
\end{lemma}

From Lemma \ref{lem:ext_cases_rank} we immediately get the next result. 

\begin{corollary}\label{cor:cyc_cl}
    Let $M_i=(E_i,r_i)$ be $q$-matroids for $i=1,2$. Let $E=E_1\oplus E_2$. Let $X,Y\in \mL(E)$. Then the following hold. 
    \begin{enumerate}
    \setlength\itemsep{0.5em}
        \item If $X=\overline{X}\oplus \boldsymbol{0}$ and $Y=\overline{Y}\oplus \boldsymbol{0}$ for some $\overline{X},\overline{Y}\in\mL(E_1)$, then  $\cl_r(X+Y)= \cl_{r_1}(\overline{X}+ \overline{Y}) \oplus \zero$ and $\cyc_r(X\cap Y)=  \cyc_{r_2}(\overline{X}\cap \overline{Y})\oplus \zero$.
        \item If $X=E_1\oplus \overline{X}$ and $Y=E_1\oplus \overline{Y}$, for some $\overline{X}, \overline{Y}\in\mL(E_2)$, then $\cl_r(X+Y)= E_1\oplus \cl_{r_2}(\overline{X}+ \overline{Y})$ and $\cyc_r(X\cap Y)= E_1\oplus \cyc_{r_2}(\overline{X}\cap \overline{Y})$.
        \item  If $X=\overline{X}\oplus \boldsymbol{0}$ and $Y=E_1\oplus \overline{Y}$ for some $\overline{X}\in\mL(E_1)$ and $\overline{Y}\in\mL(E_2)$, then $\cl_r(X+Y) = E_1\oplus \cl_{r_2}(\overline{Y})$ and $\cyc_r(X\cap Y) = \cyc_{r_1}(\overline{X})\oplus \zero$.
    \end{enumerate}
     
\end{corollary}

\subsection{The cyclic flats of the free product}\label{sec:cyc_flats_fp}
In this subsection, we describe the lattice structure of the cyclic flats of the free product of two $q$-matroids, in terms of the cyclic flats of the two factors. In order to do this, we recall some basic properties of the cyclic flats of a $q$-matroid. For more details we refer the interested reader to \cite{alfarano2022cyclic}.

We continue to use Notation~\ref{not:section_3}. Furthermore, we let $\mathcal{Z}_1$ and $\mathcal{Z}_2$ be the cyclic flats of $M_1$ and~$M_2$ respectively.
Let \begin{gather}
\mZ:=\iota_1(\mZ_1\setminus \{E_1\}) \cup \{E_1\oplus Z : Z\in\mZ_2\setminus \{\zero\}\}\subseteq\mL(E_1\oplus E_2),\label{eq:cyc_flat}\\
\mZ':=\mZ\cup \iota_1(E_1)\subseteq\mL(E_1\oplus E_2).
\end{gather}
Let $\vee$ and $\wedge$ be two operators such that for every $A,B\in\mZ$, we have $A\vee B=\cl(A+B)$ and $A\wedge B=\cyc(A\cap B)$.
We are going to show that the set $\mZ'$ is the set of cyclic flats of the free product of a coloopless $q$-matroid and a loopless $q$-matroid. If $M_1$ has coloops or $M_2$ has loops, then $\mZ$ is the collection of cyclic flats of the free product of $M_1\square M_2$. The proof consists in two steps. In Theorem \ref{thm:cyc_flats_axioms} we prove that $\mZ'$ and $\mZ$ satisfy the cyclic flat axioms from Definition \ref{def:axioms_cyclic_flats}. In Theorem \ref{thm:union_of_cycflats}, we show that either $\mZ'$ or $\mZ$ is the set of cyclic flats of $M_1\square M_2$, according to the presence of loops and coloops.

We first recall the following preliminary result.

\begin{lemma}{\cite[Lemma 2.17]{alfarano2022cyclic}}\label{lem_cyc_closure}
     Let $(E,r)$ be any $q$-matroid. Then, 
     for every $X\in\mL(E)$, we have that
    $r(X)-r(\cyc(X)) = \dim(X)-\dim(\cyc(X))$.
\end{lemma}

\begin{theorem}\label{thm:cyc_flats_axioms}
Let $r$ be the rank function of $M_1\square M_2$. If $M_1$ is coloopless and $M_2$ is loopless then $(\mZ',r)$ is a lattice of cyclic flats with respect to $r$.
If $M_1$ has coloops or $M_2$ has loops, then $(\mZ,r)$ is a lattice of cyclic flats with respect to $r$. 
\end{theorem}
\begin{proof} 
For $Z_1,Z_2 \in \mZ$ we define $Z_1 \vee Z_2:= \cl_r(Z_1 + Z_2)$ and $Z_1 \wedge Z_2:=\cyc_r(Z_1 \cap Z_2)$.
We first assume that $M_1$ does not have a coloop and $M_2$ has a loop. With this assumption,~$E_1$ is a cyclic flat of $M_1$ and $\boldsymbol{0}$ is not a cyclic flat of $M_2$. The proof of the case where $M_1$ has a coloop and $M_2$ does not have a loop is done in a very similar way, so we omit it.
\begin{enumerate}
    \item[\ref{z0}] $(\mZ,\leq,\vee,\wedge)$ is a lattice. To see this, let $Z_1,Z_2\in\mZ$. We distinguish three cases.
    \begin{itemize}
        \item  If $Z_1,Z_2\in\iota_1(\mZ_1)$, then for $i=1,2$, $Z_i=\overline{Z}_i\oplus \boldsymbol{0}$, with $\overline{Z}_i\ne E_1$. 
        In this case we have nothing to show since $\mZ_1$ is a lattice and $Z_1\vee Z_2 = (\overline{Z}_1\vee_{\mZ_1}\overline{Z}_2)\oplus \boldsymbol{0}$ and $Z_1\wedge Z_2 =(\overline{Z}_1\wedge_{\mZ_1}\overline{Z}_2)\oplus \boldsymbol{0}$. If $\bar{Z}_1\vee \bar{Z}_2=E_1$, then $Z_1\vee Z_2=E_1\oplus\iota_2(\boldsymbol{0}_{\mathcal{Z}_2})$.
        \item If $Z_i=E_1\oplus \overline{Z}_i$, with $\overline{Z}_i\in\mZ_2$ for $i=1,2$, then by Corollary~\ref{cor:cyc_cl}
        \begin{align*}
            Z_1\vee Z_2&=(E_1\oplus \overline{Z}_1)\vee(E_1\oplus \overline{Z}_2)= E_1\oplus (\overline{Z}_1\vee_{\mZ_2} \overline{Z}_2),\\
             Z_1\wedge Z_2&=(E_1\oplus \overline{Z}_1)\wedge(E_1\oplus \overline{Z}_2)= E_1\oplus (\overline{Z}_1\wedge_{\mZ_2} \overline{Z}_2).
        \end{align*}
        Since $\mZ_2$ is a lattice, we have that $(\overline{Z}_1\vee_{\mZ_2} \overline{Z}_2), \;(\overline{Z}_1\wedge_{\mZ_2} \overline{Z}_2)\in\mZ_2$, then $Z_1\vee Z_2$ and $Z_2\wedge Z_2$ belong to $\mZ$.
        \item Let $Z_1=\overline{Z}_1\oplus \boldsymbol{0}$, with $\overline{Z}_1\in\mZ_1$ and let $Z_2=E_1\oplus \overline{Z}_2$, with $\overline{Z}_2\in\mZ_2$.
        Then
        \begin{align*}
            Z_1\vee Z_2&= (\overline{Z}_1\oplus  \boldsymbol{0})\vee (E_1\oplus \overline{Z}_2) = E_1\oplus \overline{Z}_2\in\mZ,\\
            Z_1\wedge Z_2&= (\overline{Z}_1\oplus \boldsymbol{0})\wedge (E_1\oplus \overline{Z}_2) = \overline{Z}_1\oplus \boldsymbol{0}\in\mZ.
        \end{align*}
        Moreover notice that $\boldsymbol{0}_{\mZ}=\boldsymbol{0}_{\mZ_1}\oplus \boldsymbol{0}$ and $\boldsymbol{1}_{\mZ}=E_1\oplus \boldsymbol{1}_{\mZ_2}$.
    \end{itemize}
    \item[\ref{z1}] This is clearly satisfied since $r(\boldsymbol{0}_\mZ)=r(\boldsymbol{0}_{\mZ_1}\oplus \boldsymbol{0}) = 0$.
    \item[\ref{z2}] Let $Z_1,Z_2\in\mZ$. If they both belong to $\iota_1(\mZ_1)$ or $E_1\oplus \mZ_2$ then there is nothing to show. The only case to prove is when $Z_1=\overline{Z}_1\oplus \boldsymbol{0}$ and $Z_2=E_1\oplus \overline{Z}_2$, with $\overline{Z}_1\in\mZ_1$ and $\overline{Z}_2\in\mZ_2$. Note that $Z_1< Z_2$. Then by Lemma~\ref{lem:ext_cases_rank}
    \begin{align*}
        r(Z_2)-r(Z_1)&= r(E_1\oplus \overline{Z}_2)-r(\overline{Z}_1\oplus \boldsymbol{0}) \\
        &= r_1(E_1)+r_2(\overline{Z}_2)-r_1(\overline{Z}_1) \\
        &< \dim(E_1) -\dim(\overline{Z}_1)+r_2(\overline{Z}_2)\\ 
        &\leq \dim(E_1) -\dim(\overline{Z}_1)+\hh(\overline{Z}_2) \\
        &=\dim(E_1\oplus \overline{Z}_2)-\dim(\overline{Z_1}) \\
        &=\dim(Z_2)-\dim(Z_1),
    \end{align*}
    where the strict inequality follows since \ref{z2} holds in $\mZ_1$.
    \item[\ref{z3}] By the submodularity of $r$, we have that for every $F,G\in\mZ$
    \begin{align*}
        r(F)+r(G) &\geq r(F+G)+r(F\cap G).
    \end{align*}
    Now, \ref{z3} follows directly from the fact that $r(F+G)=r(F\vee G)$ and Lemma \ref{lem_cyc_closure} applied to~$F\cap G$.
\end{enumerate}

We assume now that $M_1$ has no coloops and $M_2$ has no loops. In this case, $E_1\in\mZ_1$ and $\boldsymbol{0}\in\mZ_2$. We want to show that the collection $\mZ'=\mZ\cup \iota_1(E_1)$ satisfies the axioms \ref{z0}--\ref{z3}. Note that \ref{z1} and \ref{z3} can be proved as before. 
\begin{enumerate}
\item[\ref{z0}] Let $Z\in\iota_1(\mZ_1\setminus \{\zero\})$. Then $Z\vee \iota_1(E_1)=\iota_1(E_1)\in\mZ'$ and $Z\wedge \iota_1(E_1)=Z\in\mZ'$. If $Z\in E_1\oplus \mZ_2\setminus \{\zero\}$, then $Z\vee \iota_1(E_1)=Z\in\mZ'$ and $Z\wedge \iota_1(E_1)=\iota_1(E_1)\in\mZ'$. The other cases have already been covered before. Moreover, also in this case $\boldsymbol{0}_{\mZ'}=\boldsymbol{0}_{\mZ_1}\oplus \boldsymbol{0}$ and $\boldsymbol{1}_{\mZ'}=E_1\oplus \boldsymbol{1}_{\mZ_2}$.     
\item[\ref{z2}] Let $Z_1,Z_2\in\mathcal{Z}'$ be such that $Z_1<Z_2$. As before, the only non-trivial case to show is when $Z_1\in\iota_1(\mathcal{Z}_1)\backslash(E_1\oplus\boldsymbol{0})$ and $Z_2\in E_1\oplus\mathcal{Z}_2$. However, in this case, we have that $Z_1,Z_2\in\mathcal{Z}$, so the proof follows from the previous case for \ref{z2}.
\end{enumerate}
The case where $M_1$ has a coloop and $M_2$ has a loop is proved in a similar way to the above. The only difference is that $\iota_1(1_{\mathcal{Z}_1}),E_1\oplus\iota_2(\boldsymbol{0}_{\mathcal{Z}_2})\in\mathcal{Z}$ and $\iota_1(1_{\mathcal{Z}_1})<E_1\oplus\boldsymbol{0}<E_1\oplus\iota_2(\boldsymbol{0}_{\mathcal{Z}_2})$.
\end{proof}

\begin{remark}
    Note that if $M_1$ has no coloop and $M_2$ has no loop, then $\mZ'$ is given by ``stacking" $\mZ_1$ and $\mZ_2$ on top of each other as lattices and identifying the maximal element of $\mZ_1$ and the least element of $\mZ_2$. We give an example to illustrate the situation better.
\end{remark}

\begin{example}\label{ex:no_loop_coloop}
    Let $\F=\F_2$. Consider the following three $q$-matroids. Let $M=(\F^5,r_M)$ be the $q$-matroid over the ground space $\F^5$ with collection of cyclic flats $\{F_0,\ldots,F_4\}$, where
    \begin{gather*}
        F_0=\langle e_1+e_3+e_5\rangle, \;\;  F_1=\langle e_1+e_5, e_2+e_4+e_5, e_3 \rangle,\\
        F_2=\langle e_1+e_4, e_2+e_4,  e_3+e_4+e_5\rangle, \;\;
        F_3=\langle e_1+e_3,e_2,e_5\rangle, \;\; F_4=\F^5,
    \end{gather*}
    and $r_M(F_0)=0$, $r_M(F_i)=1$ for $i=1,2,3$ and $r_M(F_4)=2$.
    Let $N=(\F^8,r_N)$ be the $q$-matroid over $\F^8$ given in \cite[Example 4.7] {gluesing2023decompositions}. Its cyclic flats are given by $\{G_0,\ldots,G_4\}$, where
    $$ G_0=\zero,\ G_1=\langle e_1,e_2\rangle,\ G_2=\langle e_1,e_2,e_3,e_4\rangle,\ G_3=\langle e_5,e_6,e_7,e_8\rangle,\ G_4=\F^8,$$
    and $r_N(G_i)=i$ for $i=0,1,2,3,4$.
    Let $L$ be the $q$-matroid $(\F^5,r_L)$, whose cyclic flats are given by $\{Z_0,\ldots,Z_4\}$, where
    \begin{gather*}
        Z_0=\zero, \; Z_1=\langle e_1+e_4+e_5,e_2+e_4\rangle, \;
        Z_2=\langle  e_1+e_3,e_4\rangle, \\
        Z_3=\langle e_3+e_5, e_1+e_2+e_4+e_5\rangle, \; Z_4=\langle e_1+e_5, e_2,e_3+e_4, e_4\rangle,
    \end{gather*}
    with $r_L(Z_0)=0, \; r_L(Z_1)=r_L(Z_2)=r_L(Z_3)=1$ and $r_L(Z_4)=2$.
    In particular, $M$ has no coloop and $N$ has no loops. Hence $\mZ' := (\mZ(M)\oplus \boldsymbol{0})\cup (\F^5\oplus\mZ(N))$ is as shown in Figure \ref{fig:Example1.1}.
    Since $L$ has a coloop and $M$ has a loop, we have that 
    $\mZ := (\mZ(L)\oplus \boldsymbol{0})\cup (\F^5\oplus\mZ(M))$ is as shown in Figure \ref{fig:Example1.2}.

\centering
\begin{figure}[ht!]
\begin{minipage}{.49\textwidth}
\centering
   \begin{tikzpicture}[x=1.4cm,y=1.8cm]
\node at (0,0)    (F0)  {$F_0\oplus \boldsymbol{0}$};
\node at (-1,1)   (F1) {$F_1\oplus \boldsymbol{0}$};
\node at (0,1)   (F2) {$F_2\oplus \boldsymbol{0}$};
\node at (1,1)   (F3) {$F_3\oplus \boldsymbol{0}$};
\node at (0,2)   (F4) {$\F^5\oplus \boldsymbol{0}$};
\node at (1,3)   (G1) {$\F^5\oplus G_1$};
\node at (2,4)   (G2) {$\F^5\oplus G_2$};
\node at (-1,4)   (G3) {$\F^5\oplus G_3$};
\node at (0,5)   (G4) {$\F^5\oplus \F^8$};
\draw (F0) -- (F1);
\draw (F0) -- (F2);
\draw (F0) -- (F3);
\draw (F1) -- (F4);
\draw (F2) -- (F4);
\draw (F3) -- (F4);
\draw (F4) -- (G1);
\draw (F4) -- (G3);
\draw (G1) -- (G2);
\draw (G3) -- (G4);
\draw (G2) -- (G4);
\end{tikzpicture}
    \caption{Lattice of cyclic flats $\mZ'$ from Example \ref{ex:no_loop_coloop}.}
    \label{fig:Example1.1}
\end{minipage}
\begin{minipage}{.49\textwidth}
\centering
   \begin{tikzpicture}[x=1.7cm,y=1.8cm]
\node at (0,0)    (Z0)  {$\boldsymbol{0}$};
\node at (-1,1)   (Z1) {$Z_1\oplus \boldsymbol{0}$};
\node at (0,1)   (Z2) {$Z_2\oplus \boldsymbol{0}$};
\node at (1,1)   (Z3) {$Z_3\oplus \boldsymbol{0}$};
\node at (0,2)   (Z4) {$Z_4\oplus \boldsymbol{0}$};
\node at (0,3)   (F0) {$\F^5\oplus F_0$};
\node at (-1,4)   (F1) {$\F^5\oplus F_1$};
\node at (0,4)   (F2) {$\F^5\oplus F_2$};
\node at (1,4)   (F3) {$\F^5\oplus F_3$};
\node at (0,5)   (F4) {$\F^5\oplus \F^5$};
\draw (Z0) -- (Z1);
\draw (Z0) -- (Z2);
\draw (Z0) -- (Z3);
\draw (Z1) -- (Z4);
\draw (Z2) -- (Z4);
\draw (Z3) -- (Z4);
\draw (Z4) -- (F0);
\draw (F0) -- (F1);
\draw (F0) -- (F2);
\draw (F0) -- (F3);
\draw (F1) -- (F4);
\draw (F2) -- (F4);
\draw (F3) -- (F4);
\end{tikzpicture}
    \caption{Lattice of cyclic flats $\mZ$ of Example \ref{ex:no_loop_coloop}.}
    \label{fig:Example1.2}
\end{minipage}
\end{figure}

\end{example}

We now describe the lattice of cyclic flats of $M_1 \square M_2$ for a pair of $q$-matroids $M_1$ and $M_2$. This is the $q$-analogue of \cite[Proposition 6.1]{crapo2005unique}, however our approach is different. 

\begin{theorem}\label{thm:union_of_cycflats}
Let $M_1$ and $M_2$ be $q$-matroids over the ground spaces $E_1$ and $E_2$, respectively. Then
$$\mathcal{Z}(M_1\square M_2)=
\begin{cases}
    \mathcal{Z} \cup \iota_1(E_1) & \textnormal{ if $M_1$ has no coloops and $M_2$ has no loops,} \\
    \mZ & \textnormal{ otherwise},
\end{cases}$$
where $\mZ$ is defined as in \eqref{eq:cyc_flat}.
\end{theorem}
\begin{proof}
As before, we will denote the rank function of $M_1\square M_2$ by $r$. By Theorem \ref{thm:cyc_flats_axioms} there exists a $q$-matroid $L=(E_1\oplus E_2,r_L)$ such that $\mathcal{Z}(L)=\mathcal{Z}$. We will show that $r=r_L$. Clearly, we have that $L|E_1\cong M_1$ and $L/E_1\cong M_2$. Let $X\in\mathcal{L}(E_1\oplus E_2)$. By Lemma \ref{lem:rank_function_cyclic_flats}, we have that
$$r_L(X)=\textup{min}\{r_L(Z)+\hh((X+Z)/Z):Z\in\mathcal{Z}(L)\}.$$
Suppose that $Z'\in\mZ(L)$ is such that $r_L(X)=r_L(Z')+\hh(X/X\cap Z')$. If $Z'\leq E_1\oplus \boldsymbol{0}$, then there exists $\overline{Z}'\in\mZ_1$ such that $Z'=\iota_1(\overline{Z}')$ and $r_L(Z')=r_1(\overline{Z}')$ by Lemma~\ref{lem:ext_cases_rank}. Therefore,
{\small{\begin{align*}
r_L(Z')+\hh(X/X\cap Z')&=r_1(\overline{Z}')+\hh(X/X\cap Z')\\
&=r_1(\overline{Z}')+\hh(X)-\hh(X\cap Z')\\
&=r_1(\overline{Z}')+\hh(X\cap(E_1\oplus \boldsymbol{0})) + \hh(\pi_2(X)) -\hh(X\cap Z') - \hh(\pi_2(Z'\cap X))\\
&=r_1(\overline{Z}') +\hh((X\cap (E_1\oplus \boldsymbol{0}))/(Z'\cap X)) +r_2(\pi_2(X))+\nu_2(\pi_2(X))\\
&= r_1(\pi_1(X\cap( E_1\oplus \boldsymbol{0})))+r_2(\pi_2(X))+\nu_2(\pi_2(X)),
\end{align*}}}
Where the third equality follows from Lemma~\ref{lem:rank_nullity}. If $Z'\geq E_1 \oplus \boldsymbol{0}$, then $Z'=E_1\oplus \overline{Z}$, with $\overline{Z}\in\mZ_2$. Then by Lemma~\ref{lem:ext_cases_rank}, we have
{\small{\begin{align*}
r_L(Z')+\hh(X/X\cap Z')&=r_L(E_1\oplus \overline{Z}) + \hh(X/X\cap Z')\\
&=r_1(E_1)+r_2(\overline{Z})+\hh(X/X\cap Z')\\
&=r_1(E_1)+r_2(\overline{Z})+\hh(X)-\hh(X\cap Z')\\
&=r_1(E_1)+r_2(\pi_2(Z'))+\hh(X)-\hh(X\cap Z')\\
&=r_1(\pi_1(X\cap (E_1\oplus \boldsymbol{0}))+\lambda_1(\pi_1(X\cap (E_1\oplus \boldsymbol{0}))\\
&\; \; \;+r_2(\pi_2(Z'))+\hh(\pi_2(X)/\pi_2(X\cap Z'))\\
&= r_1(\pi_1(X\cap (E_1\oplus \boldsymbol{0}))+\lambda_1(\pi_1(X\cap (E_1\oplus \boldsymbol{0}))+r_2(X).
\end{align*}}}

By Theorem \ref{thm:free_prod_rank_function}, we get that $r_L=r$. 
\end{proof}

We can characterise the free product of two uniform $q$-matroids in terms of cyclic flats as follows.

\begin{theorem}\label{thm:char_uniform}
    Let $k_1,k_2,n_1,n_2$ be integers and let $k=k_1+k_2, n=n_1+n_2$, $0<k<n$. Let $M=(\F_q^n,r)$ be a $q$-matroid of rank $k$. The following are equivalent.
    \begin{enumerate}
        \item $M\cong\mU_{k_1,n_1}(q)\square\mU_{k_2,n_2}(q)$.
        \item  $\mZ(M)=\{\boldsymbol{0}, \F_q^{n_1}\oplus \boldsymbol{0}, \F_q^n\}$ whose corresponding ranks are equal to $0,k_1,k$, respectively.
    \end{enumerate}
\end{theorem}
\begin{proof}
     The results easily follows from Lemma \ref{lem:cyc_flats_uniform} and Theorem \ref{thm:union_of_cycflats}.
\end{proof}


\section{Weak ordering and unique factorisation}\label{sec:weak ordering}

In section~\ref{subsec:weak_order} we establish that the free product of $q$-matroids $M_1,\dots,M_k$ on ground spaces $E_1,\dots,E_k$ respectively is the unique (up to isomorphism) $q$-matroid with the maximum number of independent spaces among all $q$-matroids $N$ on $E_1 \oplus\dots\oplus E_k$ such that $N[E_{i-1},E_i]\cong M_i$ for $i\in \{1,\dots,k\}$. This gives the $q$-analogue of one of the main results of \cite[section 4]{crapo2005unique}. Furthermore, in \cite{crapo2005unique} it is noted that the direct sum of matroids is the ``most dependent" way of combining two matroids (in a sense that we will make precise later). 
We will show that, somewhat surprisingly, the $q$-analogue of this result does not hold.

In section~\ref{subsec:uniq_fact} we show that any $q$-matroid can be factorised uniquely (up to isomorphism) into irreducible components, which gives the $q$-analogue of another one of the main results found in \cite[section 6]{crapo2005unique}.

Before proceeding with sections \ref{subsec:weak_order} and \ref{subsec:uniq_fact} we recall the definitions and some basic properties of the direct sum of $q$-matroids and we establish some relevant preliminary results.

\begin{definition}
    Let $M_i=(E_i,r_i),\,i=1,2,$ be a pair of $q$-matroids and set $E=E_1\oplus E_2$.\\
For $i=1,2$, let $r_i'$ be the map defined by $r'_i:\mL(E)\longrightarrow \mathbb{N}_0,\ V\longmapsto r_i(\pi_i(V))$ for all $V \in \mL(E)$.
Let $r$ be the map defined by:
\begin{equation*}
r:\mL(E)\longrightarrow\mathbb{N}_0,\quad V\longmapsto \dim V+\min_{X\leq V}\big(r'_1(X)+r'_2(X)-\dim X\big)\:\:\forall \:\:V \in \mL(E).
\end{equation*}
Then $M:=(E,r)$ is a $q$-matroid, called the \textbf{direct sum of $M_1$ and~$M_2$}, and is denoted by $M_1\oplus M_2$.
\end{definition}

\begin{lemma}[{\cite[Theorem 6.2]{gluesing2023decompositions}}]\label{lem:sum_cyclic_flats}
    For each $i=1,2$, let $M_i=(E_i,\rho_i)$, be a $q$-matroid and let $\mathcal{Z}_i$ be the lattice of cyclic flats of $M_i$. 
    Let $\mZ_1\oplus\mZ_2=\{Z_1\oplus Z_2: Z_1\in\mZ_1, \; Z_2\in\mZ_2\}$. Then 
    $$\mZ(M_1\oplus M_2)=\mZ_1\oplus\mZ_2.$$
\end{lemma}

It is well-known that the dual of a $q$-matroid is unique up to lattice-equivalence, regardless of what lattice anti-isomorphism is used.
We will define a lattice anti-isomorphism $\varphi$ that is convenient for our notation, but the reader could be aware that, with the use of additional lattice isomorphisms, the following results that use this specific lattice anti-isomorphism $\varphi$ can hold for arbitrary lattice anti-isomorphisms (but the statements would include slightly heavier notation). 

    \begin{notation}\label{not:perp}
        We fix $\langle \cdot , \cdot \rangle_{E_1}$ and $\langle \cdot , \cdot \rangle_{E_2}$ to be arbitrary non-degenerate inner products on $E_1$ and~$E_2$ respectively. We let 
        $\langle \cdot , \cdot \rangle=\langle \cdot , \cdot \rangle_{E_1} \oplus \langle \cdot , \cdot \rangle_{E_2}$ be the non-degenerate inner product on $E_1 \oplus E_2$
        defined by $\langle a , b  \rangle=\langle  \pi_1(a), \pi_1(b) \rangle_{E_1} + \langle \pi_2(a) , \pi_2(b) \rangle_{E_2}$ for all $a,b \in E_1 \oplus E_2$.
        For each $V \leq E_1\oplus E_2$, we let $V^\perp$ denote the orthogonal complement of $V$ with respect to $\langle \cdot , \cdot \rangle$.
        Then $(E_1\oplus\boldsymbol{0})^\perp =\boldsymbol{0}\oplus E_2$ and $(\boldsymbol{0}\oplus E_2)^\perp = E_1\oplus\boldsymbol{0}$. 
    \end{notation}

    \begin{definition}\label{def:dual_isom}
        We define the function $\text{rev}:\mathcal{L}\rightarrow\mathcal{L}$ by $\text{rev}(X)=\{(x_n,\dots,x_1):(x_1,\dots,x_n)\in X\}$, which is a lattice automorphism. Now we define the lattice anti-isomorphism $\varphi:\mathcal{L}\rightarrow\mathcal{L}$ by $\varphi(X)=\text{rev}(X^\perp)$. Then for $\mathcal{L}=\mathcal{L}(E_1\oplus E_2)$ and $\perp$ as in Notation~\ref{not:perp}, we have $\varphi(E_1\oplus\boldsymbol{0})=E_2\oplus\boldsymbol{0}$.
    \end{definition}

    For the remainder of this section, we will use the lattice anti-isomorphism $\varphi$ to map a $q$-matroid to its dual. The following duality result is a $q$-analogue of \cite[Proposition 4]{crapo2005free}, which the authors prove via the bases of a matroid. We prove it by examining the lattice of cyclic flats. 

    \begin{proposition}\label{prop:dual_free_prod}
        Let $M_1$ and $M_2$ be $q$-matroids on the spaces $E_1$ and $E_2$ respectively. We have that $(M_1\square M_2)^*\cong M_2^*\square M_1^*$.
    \end{proposition}
    \begin{proof}
        Recall from Definition~\ref{def:dual_isom} that $\varphi(E_1\oplus\boldsymbol{0})=E_2\oplus \boldsymbol{0}$. We will prove the statement of this proposition by considering the lattices of cyclic flats $\mathcal{Z}((M_1\square M_2)^*)$ and $\mathcal{Z}(M_2^*\square M_1^*)$. Recall that the dual of a cyclic flat of a $q$-matroid is a cyclic flat of the dual $q$-matroid.
        We have the following:
        \begin{align*}
            Z\in\mathcal{Z}((M_1\square M_2)^*)&\iff \varphi(Z)\in\mathcal{Z}(M_1\square M_2)\\
            &\iff \varphi(Z)\in\mathcal{Z}((M_1\square M_2)|(E_1\oplus\boldsymbol{0}))\text{ or }\varphi(Z)\in\mathcal{Z}((M_1\square M_2)/(E_1\oplus\boldsymbol{0}))\\
            &\iff Z\in\mathcal{Z}((M_1\square M_2)^*|\varphi(E_1\oplus\boldsymbol{0}))\text{ or }Z\in\mathcal{Z}((M_1\square M_2)^*/\varphi(E_1\oplus\boldsymbol{0}))\\
            &\iff Z\in\mathcal{Z}((M_1\square M_2)^*|(E_2\oplus\boldsymbol{0}))\text{ or }Z\in\mathcal{Z}((M_1\square M_2)^*/(E_2\oplus\boldsymbol{0})).
        \end{align*}
        Using Lemma~\ref{lem:duality_properties}, we have 
        $$(M_1\square M_2)^*|\varphi(E_1\oplus\boldsymbol{0})\cong((M_1\square M_2)/(E_1\oplus\boldsymbol{0}))^*$$
        and
        $$(M_1\square M_2)^*/\varphi(E_1\oplus\boldsymbol{0})\cong((M_1\square M_2)|(E_1\oplus\boldsymbol{0}))^*.$$
        This gives us that
        $$(M_1\square M_2)^*|\varphi(E_1\oplus\boldsymbol{0})\cong M_2^*\cong(M_2^*\square M_1^*)|\varphi(E_1\oplus\boldsymbol{0})\cong (M_2^*\square M_1^*)|(E_2\oplus\boldsymbol{0})$$
        and
        $$(M_1\square M_2)^*/\varphi(E_1\oplus\boldsymbol{0})\cong M_1^*\cong(M_2^*\square M_1^*)/\varphi(E_1\oplus\boldsymbol{0})\cong(M_2^*\square M_1^*)/(E_2\oplus\boldsymbol{0}).$$
        We obtain the result by applying Theorem~\ref{thm:union_of_cycflats} to $M_2^*\square M_1^*$.
    \end{proof}

  \begin{proposition}\label{prop:assoc_dir_prod}
        The free product on $q$-matroids is associative.
    \end{proposition}

    \begin{proof}
        For $i=1,2,3$, let $M_i$ be a $q$-matroid with ground space $E_i$. Consider the projection $\pi_i : E_1\oplus E_2\oplus E_3\longrightarrow E_i$ the projection on the corresponding coordinates.
        By Theorem~\ref{thm:union_of_cycflats} we have that $Z\in\mathcal{Z}((M_1\square M_2)\square M_3)$ if and only if $Z\in[\boldsymbol{0},E_1\oplus E_2\oplus\boldsymbol{0}]\cup[E_1\oplus E_2\oplus\boldsymbol{0},E_1\oplus E_2\oplus E_3]$ and 
        \begin{itemize}
         \setlength\itemsep{0.5em}
        \item[(1)] $\pi_{1}(Z)\oplus \pi_2(Z)\in\mathcal{Z}(M_1\square M_2)$ if $Z<E_1\oplus E_2\oplus\boldsymbol{0}$,
            \item[(2)] $\pi_{3}(Z)\in\mathcal{Z}(M_3)$ if $Z>E_1\oplus E_2\oplus\boldsymbol{0}$,
            \item[(3)] $\pi_{1}(Z)\oplus \pi_2(Z)\in\mathcal{Z}(M_1\square M_2)$ and $\pi_{3}(Z)\in\mathcal{Z}(M_3)$ if $Z=E_1\oplus E_2\oplus\boldsymbol{0}$.
        \end{itemize}
    Now consider that $Z'\in\mathcal{Z}(M_1\square M_2)$ if and only if $Z'\in[\boldsymbol{0},E_1\oplus\boldsymbol{0}]\cup[E_1\oplus\boldsymbol{0},E_1\oplus E_2]$ and
    \begin{itemize}
     \setlength\itemsep{0.5em}
        \item[(4)] $\pi_{1}(Z')\in\mathcal{Z}(M_1)$ if $Z'<E_1\oplus\boldsymbol{0}$,
        \item[(5)] $\pi_{2}(Z')\in\mathcal{Z}(M_2)$ if $Z'>E_1\oplus\boldsymbol{0}$,
        \item[(6)] $\pi_{1}(Z')\in\mathcal{Z}(M_1)$ and $\pi_{2}(Z')\in\mathcal{Z}(M_2)$ if $Z'=E_1\oplus\boldsymbol{0}$.
    \end{itemize}
    Note that $\pi_{i}(\pi_{1}(Z)\oplus \pi_2(Z))=\pi_{i}(Z)$ for $i=1,2$. Observe that (4), (5) and (6) imply (1). Consider the following property, which is implied by (3) and (6).
    \begin{itemize}
        \item[(3')] $\pi_{2}(Z)\in\mathcal{Z}(M_2)$ and $\pi_{3}(Z)\in\mathcal{Z}(M_3)$ if $Z=E_1\oplus E_2\oplus\boldsymbol{0}$.
    \end{itemize}
    We see that (3') and (6) imply (3).
    We therefore have that $Z\in\mathcal{Z}((M_1\square M_2)\square M_3)$ if and only if 
    $$Z\in[\boldsymbol{0},E_1\oplus\boldsymbol{0}\oplus\boldsymbol{0}]\cup[E_1\oplus\boldsymbol{0}\oplus\boldsymbol{0},E_1\oplus E_2\oplus\boldsymbol{0}]\cup[E_1\oplus E_2\oplus\boldsymbol{0},E_1\oplus E_2\oplus E_3]$$
    and one of the above properties (2),(3'),(4),(5),(6) is satisfied.

    With a similar line of reasoning we can show the same properties for $Z\in\mathcal{Z}(M_1\square(M_2\square M_3))$ which shows that  $\mathcal{Z}((M_1\square M_2)\square M_3) = \mathcal{Z}(M_1\square(M_2\square M_3))$.

    For any $[X,Y]\subseteq[\boldsymbol{0},E_1\oplus E_2\oplus E_3]$, the rank function of the minor $((M_1\square M_2)\square M_3)[X,Y]$ is $r_{(M_1\square M_2)\square M_3}(W)-r_{(M_1\square M_2)\square M_3}(X)$ for any $W\in[X,Y]$. Similarly, the rank function of the minor $(M_1\square (M_2\square M_3))[X,Y]$ is $r_{M_1\square (M_2\square M_3)}(W)-r_{M_1\square (M_2\square M_3)}(X)$ for any $W\in[X,Y]$. Using the definition of the free product and Lemma \ref{lem:ext_cases_rank}, we can thus deduce that
    \begin{align*}
        ((M_1\square M_2)\square M_3)[\boldsymbol{0},E_1\oplus\boldsymbol{0}\oplus\boldsymbol{0}]&\cong M_1\cong(M_1\square (M_2\square M_3))[\boldsymbol{0},E_1\oplus \boldsymbol{0}\oplus\boldsymbol{0}],\\
        ((M_1\square M_2)\square M_3)[E_1\oplus\boldsymbol{0}\oplus\boldsymbol{0},E_1\oplus E_2\oplus\boldsymbol{0}]&\cong M_2\cong (M_1\square (M_2\square M_3))[E_1\oplus\boldsymbol{0}\oplus\boldsymbol{0},E_1\oplus E_2\oplus\boldsymbol{0}],\\
        ((M_1\square M_2)\square M_3)[E_1\oplus E_2\oplus\boldsymbol{0},E_1\oplus E_2\oplus E_3]&\cong M_3\cong(M_1\square (M_2\square M_3))[E_1\oplus E_2\oplus\boldsymbol{0},E_1\oplus E_2\oplus E_3].
    \end{align*}
    We have already shown that all cyclic flats in $\mathcal{Z}((M_1\square M_2)\square M_3) = \mathcal{Z}(M_1\square(M_2\square M_3))$ lie in these intervals. Therefore, $r_{(M_1\square M_2)\square M_3}(Z)=r_{M_1\square(M_2\square M_3)}(Z)$ for all $Z\in\mathcal{Z}((M_1\square M_2)\square M_3) = \mathcal{Z}(M_1\square(M_2\square M_3))$. Since every $q$-matroid is uniquely determined by its lattice of cyclic flats along with their rank values, this implies that $(M_1\square M_2)\square M_3\cong M_1\square (M_2\square M_3)$.
    \end{proof}

The free product and the direct sum of $q$-matroids $M_1$ and $M_2$ on ground spaces $E_1$ and $E_2$ are related by the following result.
\begin{proposition}
Let $M_1=(E_1,r_1)$ and $M_2=(E_2,r_2)$. Assume $r_1(M_1)=0$ or $\nu_2(M_2)=0$, then $M_1\square M_2=M_1\oplus M_2$.
\end{proposition}
\begin{proof}
If $r_1(M_1)=0$, then $\mZ_1=\{E_1\}$ and hence the lattice of cyclic flats of $M_1\square M_2$ is equal to $$\{E_1\oplus Z:Z\in\mathcal{Z}_2\}=\{Z_1\oplus Z_2:Z_1\in\mZ_1,Z_2\in\mathcal{Z}_2\}.$$
Similarly, if $\nu_2(M_2)=0$, then $\mZ_2=\{\boldsymbol{0}\}$, hence the lattice of cyclic flats of $M_1\square M_2$ is equal to $$\{Z\oplus \boldsymbol{0}:Z\in\mZ_1\}=\{Z_1\oplus Z_2:Z_1\in\mZ_1,Z_2\in\mathcal{Z}_2\}.$$
By Lemma \ref{lem:sum_cyclic_flats}, the result follows.
\end{proof}

\subsection{The weak ordering}\label{subsec:weak_order}

In \cite{crapo2005unique} the free product of matroids, as well as the direct sum of matroids, are discussed relative to weak maps of matroids. In the same paper, it is shown that
the free product of two matroids $M$ and $N$ on ground sets $S$ and $T$ is the maximal element in the \emph{weak ordering} of the  class of matroids on $S \cup T$ whose restrictions to $S$ and complementary contraction by $S$ are $M$ and $N$, respectively. Further, the direct sum of $M$ and $N$ is the minimal element in this weak ordering.
In this section we extend these results to $q$-matroids, as well as observe a difference that occurs in the $q$-analogue. Namely, that the direct sum is not always minimal in the weak ordering.

The definition of \emph{weak map} for matroids can be found in \cite[Chapter 9]{white_matroids}. In the following definition we include the $q$-analogue, as well as a definition of \emph{weak isomorphism}.

\begin{definition}
Let $M_1=(E_1,r_1)$ and $M_2=(E_2,r_2)$ be $q$-matroids, let $\mL_1=\mL(E_1)$ and $\mL_2=\mL(E_2)$ and let $\tau :\mathcal{L}_1\rightarrow\mathcal{L}_2$ be a function. We say that $\tau$ is a \textbf{weak map} from $M_1$ to $M_2$ if $r_1(X)\geq r_2(\tau(X))$ for every $X\in\mathcal{L}_1$. If $\tau$ is also a lattice isomorphism, then we call $\tau$ a \textbf{weak isomorphism} from $M_1$ to $M_2$.
    \end{definition}

    The following definition is the ($q$-analogue) of a \emph{weak order} on matroids; see \cite{crapo2005unique}.

    \begin{definition}
        Let $M_1=(E_1,r_1)$ and $M_2=(E_2,r_2)$ be $q$-matroids, let $\mL_1=\mL(E_1)$ and $\mL_2=\mL(E_2)$. We say that $M_2$ is below $M_2$ in the \textbf{weak ordering} and we write $M_2\preceq M_1$ if there exists a weak isomorphism from $M_1$ to $M_2$.
    \end{definition}

    It is clear that the weak order of $q$-matroids is a partial order on the set of $q$-matroids supported on the same subspace lattice (up to isomorphism). For the remainder of this section, we fix a positive integer $k$ and a collection of $q$-matroids $M_1,\dots,M_k$.

    \begin{definition}\label{def:poset}
        We denote by $\mathcal{M}_q(M_1,\dots,M_k)$ the partially ordered set of (isomorphism classes of) $q$-matroids $M$ on a vector space $E$ for which there exists a chain $\boldsymbol{0}=V_0<V_1<\dots<V_k=E$ such that $M[V_{i-1},V_i]\cong M_i$ for all $i\in\{1,\dots,k\}$, and for which the order on $\mathcal{M}_q(M_1,\dots,M_k)$ is the weak order $\preceq$. For the partially ordered set of matroids with these same conditions, we denote the corresponding poset by $\mathcal{M}(M_1,\dots,M_k)$ (in which case $M_1,\dots,M_k$ are assumed to be matroids).
    \end{definition}

    Clearly, $M_1\oplus \cdots\oplus M_k$ and $M_1\square \cdots \square M_k$ belong to $\mathcal{M}_q(M_1,\ldots,M_k)$.
    
    In Theorem~\ref{thm:free_prod_is_maximal} we generalise a result in \cite[section 4]{crapo2005unique} by showing that $M_1\square\dots\square M_k$ is maximal in $\mathcal{M}_q(M_1,\dots,M_k)$. 
    Our proof is different from the matroid case and heavily relies on the properties of cyclic flats. We first prove a technical lemma.

    \begin{lemma}\label{lem:cyccl_ineq}
        Let $M=(E,r)$ be a $q$-matroid and let $X,Y\in\mathcal{L}(E)$. We have
        $$r(Y)+\text{dim}((X+Y)/Y)\geq r(\text{cyc}(\text{cl}(Y)))+\text{dim}((X+\text{cyc}(\text{cl}(Y)))/\text{cyc}(\text{cl}(Y))).$$
    \end{lemma}
    \begin{proof}
    From the definitions of cl and cyc and Lemma \ref{lem_cyc_closure}, for any $W\in\mathcal{L}$ we have that $r(W)=r(\text{cl}(W))$ and $r(W)=r(\text{cyc}(W))+\text{dim}(W)-\text{dim}(\text{cyc}(W))$. Let $U,V\in\mathcal{L}(E)$ and $U\leq V$. Then, 
        $$\text{dim}((W+V)/V)\leq\text{dim}((W+U)/U)\leq\text{dim}((W+V)/V)+\text{dim}(V)-\text{dim}(U).$$
        We therefore get the following:
        {\small{\begin{align*}
            r(Y)+\text{dim}((X+Y)/Y)&=r(\text{cl}(Y))+\text{dim}((X+Y)/Y)\\
            &\geq r(\text{cl}(Y))+\text{dim}((X+\text{cl}(Y))/\text{cl}(Y))\\
            &=r(\text{cyc}(\text{cl}(Y)))+\text{dim}(\text{cl}(Y))-\text{dim}(\text{cyc}(\text{cl}(Y))+\text{dim}((X+\text{cl}(Y))/\text{cl}(Y))\\
            &\geq r(\text{cyc}(\text{cl}(Y)))+\text{dim}((X+\text{cyc}(\text{cl}(Y)))/\text{cyc}(\text{cl}(Y))).\qedhere
        \end{align*}}}
    \end{proof}

    The following, Theorem~\ref{thm:free_prod_is_maximal}, is the main result of this section, which is the $q$-analogue of \cite[Proposition 4.7]{crapo2005unique}. Once again, we leverage the structure of the cyclic flats of the free product to obtain a straightforward proof.
    \begin{theorem}\label{thm:free_prod_is_maximal}
        Let $M=(E,r)$ be a $q$-matroid and $\boldsymbol{0}=V_0<V_1<\dots<V_k=E$ be a chain. Let $M_i=M[V_{i-1},V_i]$. The identity map on $E$ is a weak map from $M_1\square\dots\square M_k$ to $M$.
    \end{theorem}
    \begin{proof}
        Let $N:=M_1\square\dots\square M_k$ and let $S=\{V_i:i\in\{0,\dots,k\},V_i\notin\mZ_1\}$. By Theorem~\ref{thm:union_of_cycflats} and the associativity of the free product given in Proposition~\ref{prop:assoc_dir_prod}, we have 
        $$\mathcal{Z}(N)=\left(\bigcup_{i=0}^k\mathcal{Z}(M_i)\right)\backslash S.$$
        Let $r_N$ be the rank function of $N$. 
        Since $N[V_{i-1},V_i]\cong M_i$ for all $i\in \{1,\dots,k\}$ (by the definition of the free product as well as its associativity), we have that $r(X)=r_N(X)$ for $X\in[V_{i-1},V_i]$ for $i\in  \{1,\dots,k\}$. By Lemma \ref{lem:rank_function_cyclic_flats}, for any $X\in\mathcal{L}$ we have that $r_N(X)=\min\{r_N(Z)+\text{dim}((X+Z)/Z):Z\in\mathcal{Z}(N)\}$. Let $Z'\in\mathcal{Z}(N)$ be such that $r_N(X)=r_N(Z')+\text{dim}((X+Z')/Z')$. By Theorem~\ref{thm:union_of_cycflats} and Proposition~\ref{prop:assoc_dir_prod}, we have that $\pi_i(Z')\in[V_{i-1},V_i]$ for some $i\in \{1,\dots,k\}$, which means that $r_N(Z')=r(Z')$. Therefore, by Lemma~\ref{lem:cyccl_ineq}, we have
        \begin{align*}
            r_N(X)&=r_N(Z')+\text{dim}((X+Z')/Z')\\
            &=r(Z')+\text{dim}((X+Z')/Z')\\
            &\geq r(\text{cyc}(\text{cl}(Z')))+\text{dim}((X+\text{cyc}(\text{cl}(Z')))/\text{cyc}(\text{cl}(Z')))\\
            &\geq\min\{r(Z)+\text{dim}((X+Z)/Z):Z\in\mZ_1\}=r(X).\qedhere
        \end{align*}
    \end{proof}

    The following is an immediate consequence of Theorem~\ref{thm:free_prod_is_maximal}.

    \begin{corollary}\label{cor:maximality_free_prod}
        $M_1\square\cdots\square M_k$ is maximal in $\mathcal{M}_q(M_1,\dots,M_k)$.
    \end{corollary}

    The following, Lemma~\ref{lem:cyc_flat_weak_map}, gives a convenient sufficient condition for the existence of a weak isomorphism between ($q$-)matroids.

    \begin{lemma}\label{lem:cyc_flat_weak_map}
        Let $M_1=(E,r_1)$ and $M_2=(E,r_2)$ be $q$-matroids. If $\mZ_1\subseteq\mathcal{Z}_2$ and $r_1(Z)=r_2(Z)$ for all $Z\in\mZ_1$, then the identity map on $\mathcal{L}(E)$ is a weak isomorphism from $M_1$ to $M_2$.
    \end{lemma}
    \begin{proof}
        For any $X\in\mathcal{L}(E)$, by Lemma \ref{lem:rank_function_cyclic_flats}  (\cite[Corollary 3.12]{alfarano2022cyclic}) we have
        \begin{align*}
            r_2(X)&=\textup{min}\{r_2(Z)+\textup{dim}(X/X\cap Z):Z\in\mathcal{Z}_2\}\\
            &\leq\textup{min}\{r_1(Z)+\textup{dim}(X/X\cap Z):Z\in\mZ_1\}=r_1(X)
        \end{align*}
        since $\mZ_1\subseteq\mathcal{Z}_2$. The result follows from the definition of a weak isomorphism.
    \end{proof}

    It is noted in \cite{crapo2005unique} that $M_1\oplus\dots\oplus M_k$ is minimal in $\mathcal{M}(M_1,\dots,M_k)$. Interestingly, this is not always the case for $q$-matroids, which we show in the following example.

\begin{example}\label{eg:dir_sum_not_minimal}
Let $M\cong N\cong \mathcal{U}_{1,2}(q)$. Choose a basis $\{e_1,e_2,e_3,e_4\}$ of $\mathbb{F}_q^4$ such that $(M\oplus N)|\langle e_1,e_2\rangle\cong M$ and $(M\oplus N)|\langle e_3,e_4\rangle\cong N$. We then have $\mathcal{Z}(M\oplus N)=\{\boldsymbol{0},\langle e_1,e_2\rangle,\langle e_3,e_4\rangle, \mathbb{F}_q^4\}$. Let $L$ be a $q$-matroid on $\mathcal{L}(\mathbb{F}_q^4)$ such that $\mathcal{Z}(L)=\{\boldsymbol{0},\langle e_1,e_2\rangle,\langle e_3,e_4\rangle,\langle e_1+e_3,e_2+e_4\rangle, \mathbb{F}_q^4\}$. It is not difficult to see that $\mathcal{Z}(L)$ satisfies the cyclic flat axioms from Definition \ref{def:axioms_cyclic_flats}, when we give $\langle e_1+e_3,e_2+e_4\rangle$ a rank value of 1. It is clear that $L|\langle e_1,e_2\rangle\cong M$ and $L/\langle e_1,e_2\rangle\cong N$. Furthermore, we have that $\mathcal{Z}(M\oplus N)\subset\mathcal{Z}(L)$, which by Lemma~\ref{lem:cyc_flat_weak_map} means that $L \prec M\oplus N$ in the weak order on $\mathcal{M}_q(M,N)$.
\end{example}

    \subsection{Unique factorisation}\label{subsec:uniq_fact}
    \begin{definition}
        We say that a $q$-matroid is \textbf{irreducible} if it cannot be written as the free product of two non-trivial $q$-matroids.
    \end{definition}
    A natural question that arises is whether a $q$-matroid can be decomposed uniquely into irreducible components. In this section, we answer this question in the affirmative. 

   We remark that many properties of the lattice of cyclic flats of a ($q$-)matroid do not depend on whether or not the ambient lattice is a Boolean or a subspace lattice. Some arguments that employ the cyclic flats of matroids in \cite{crapo2005unique} hold true also for $q$-matroids. Therefore, in this section, we provide only the proofs of the results that do not immediately follow as in the matroid case.
   
    \begin{definition}
        An element $A\in\mathcal{L}(E)$ is a \textbf{free separator} of a $q$-matroid $M=(E,r)$ if every cyclic flat in $\mathcal{Z}(M)$ is comparable to $A$ in $\mathcal{L}(E)$.
    \end{definition}

      Observe that, trivially, $\boldsymbol{0}$ and $E$ are free separators. We refer to any other free separator as a \emph{non-trivial} free separator.
      
    We include the following useful notation for lattices of cyclic flats. Note that it does not equate to taking the lattice of cyclic flats of minors of the $q$-matroid.

    \begin{notation}
        Let $M=(E,r)$ be a $q$-matroid. For $X\in\mathcal{L}(E)$ we will let $\mathcal{Z}(M)|X=\{Z\in\mathcal{Z}(M):Z\leq X\}$ and $\mathcal{Z}(M)/X=\{Z\in\mathcal{Z}(M):Z\geq X\}$.
    \end{notation}

    The following result is the $q$-analogue of \cite[Theorem 6.3]{crapo2005unique}. 
    We will rely heavily on our characterization result on the lattice of cyclic flats of the free product of a pair of $q$-matroids (Theorem \ref{thm:union_of_cycflats}) in order to prove it. 

    \begin{theorem}\label{thm:free_sep_free_prod}
        For any $q$-matroid $M=(E_1\oplus E_2,r)$, the following are equivalent:
        \begin{enumerate}
            \item[(i)] $M=(M|(E_1\oplus\boldsymbol{0}))\square(M/(E_1\oplus\boldsymbol{0}))$.
            \item[(ii)] $E_1\oplus\boldsymbol{0}$ is a free separator of $M$.
        \end{enumerate}
    \end{theorem}

    \begin{proof}
        \underline{(i)$\implies$(ii)} This follows from Theorem~\ref{thm:union_of_cycflats}.

        \underline{(ii)$\implies$(i)} Assume that $E_1\oplus\boldsymbol{0}$ is a free separator of $M$ and let $N=(M|(E_1\oplus\boldsymbol{0}))\square(M/(E_1\oplus\boldsymbol{0}))$. It is clear that we have $\mathcal{Z}(M)|(E_1\oplus\boldsymbol{0})\subseteq\mathcal{Z}(M|(E_1\oplus\boldsymbol{0}))$ and $\mathcal{Z}(M)/(E_1\oplus\boldsymbol{0})\subseteq\mathcal{Z}(M/(E_1\oplus\boldsymbol{0}))$. We thus have that
        \begin{equation}\label{eq:cyc_flat_subset}
            \mathcal{Z}(M)\subseteq\{\iota_1(Z):Z\in\mathcal{Z}(M|(E_1\oplus\boldsymbol{0}))\}\dot\cup\{(E_1\oplus\boldsymbol{0})+\iota_2(Z):Z\in\mathcal{Z}(M/(E_1\oplus\boldsymbol{0}))\}.
        \end{equation}
        By Theorem~\ref{thm:union_of_cycflats}, we deduce that $E_1\oplus\boldsymbol{0}\in\mathcal{Z}(N)$ if and only if $E_1\oplus\boldsymbol{0}\in\mathcal{Z}(M)$. Therefore, using (\ref{eq:cyc_flat_subset}) we conclude that $\mathcal{Z}(M)\subseteq\mathcal{Z}(N)$. By Lemma~\ref{lem:cyc_flat_weak_map} we now have that the identity map on $\mathcal{L}(E_1\oplus E_2)$ is a weak map from $M$ to $N$. Since $N|(E_1\oplus\boldsymbol{0})=M|(E_1\oplus\boldsymbol{0})$ and $N/(E_1\oplus\boldsymbol{0})=M/(E_1\oplus\boldsymbol{0})$, Theorem~\ref{thm:free_prod_is_maximal} gives us that the identity map on $\mathcal{L}(E_1\oplus E_2)$ is a weak map from $N$ to $M$. Therefore, $M=N$.
    \end{proof}

    Clearly, if a $q$-matroid $M = (E,r)$ is reducible with respect to the free product, then there exist subspaces $E_1$ and $E_2$ such that $E=E_1 \oplus E_2$ and
    $M=(M|(E_1\oplus\boldsymbol{0}))\square(M/(E_1\oplus\boldsymbol{0}))$, in which case $E_1 \oplus \zero$ is a free separator of $M$. We therefore have the following.

    \begin{corollary}\label{thm:irred_iff_no_free_sep}
        For any non-zero $q$-matroid $M$, the following are equivalent:
        \begin{enumerate}
            \item $M$ is irreducible with respect to the free product.
            \item $M$ has no non-trivial free separator.
        \end{enumerate}
    \end{corollary}
    
    \begin{definition}
        We say that a $q$-matroid is \textbf{disconnected} if it can be written as direct sum of a pair of non-trivial $q$-matroids.
    \end{definition}
    
    \begin{corollary}
        Let $M$ be loopless and coloopless. If $M$ is disconnected, then it is irreducible.
    \end{corollary}

    \begin{proof}
        If $M$ is disconnected, then there exist non-trivial $q$-matroids $M_1$ and $M_2$ on the spaces $E_1$ and $E_2$ respectively such that $M=M_1\oplus M_2$. By \cite[Proposition 13]{ceria2024direct} and its dual statement, we have that $M_1$ and $M_2$ are both loopless and coloopless. In particular, $\{\boldsymbol{0},E_1\}\subseteq\mathcal{Z}(M_1)$ and $\{\boldsymbol{0},E_2\}\subseteq\mathcal{Z}(M_2)$. By Lemma \ref{lem:sum_cyclic_flats} (\cite[Proposition 6.2]{gluesing2023decompositions}) we have that $E_1\oplus\boldsymbol{0},\boldsymbol{0}\oplus E_2\in\mathcal{Z}(M)$. By 
        Corollary~\ref{thm:irred_iff_no_free_sep} 
    the result follows.
    \end{proof}

    We recall the definition of a \emph{pinchpoint} from \cite{crapo2005unique}.

    \begin{definition}
        Let $P$ be an arbitrary poset and let $x\in P$. We say that $x$ is a \textbf{pinchpoint} if all other elements of $P$ are comparable to $x$. A pinchpoint different from $\boldsymbol{0}$ and $\boldsymbol{1}$ is called \textbf{non-trivial}. 
    \end{definition}

    We recall the following convenient definitions from \cite{crapo2005unique}, endowing them with the obvious $q$-analogue.

    \begin{definition}\label{def:D_lattice}
        Let $M=(\mathcal{L},r)$ be a $q$-matroid and $\mathcal{Z}$ its lattice of cyclic flats. Define the sublattice $\mathcal{D}(M)$ of $\mathcal{L}$ by
        $$\mathcal{D}(M):=\left\{\bigcap_{Z\in S}Z:S\subseteq \mathcal{Z}\right\}\cup\left\{\sum_{Z\in S}Z:S\subseteq\mathcal{Z}\right\}.$$
    \end{definition}

    Definition~\ref{def:D_lattice} is used to facilitate the following definition, which will be central to the decomposition of a $q$-matroid by the free product. The following is the $q$-analogue of \cite[Definition 6.14]{crapo2005unique}.

    \begin{definition}
        The \textbf{primary flag} $\mathcal{T}_M$ of a $q$-matroid $M$ is the chain $T_0<\dots< T_k$ of all pinchpoints of $\mathcal{D}(M)$.
    \end{definition}

    The following proposition is the $q$-analogue of \cite[Proposition 6.9]{crapo2005unique}, and follows immediately from Lemma~\ref{lem:cyc_flats_uniform}.
    
    \begin{proposition}\label{prop:D(M)_of_uniform}
        A non-trivial $q$-matroid $M$ is uniform if and only if $\mathcal{D}(M)=\{\zero,\one\}$.
    \end{proposition}

    The following, Theorem~\ref{thm:no_pinchp_iff_irred}, is the $q$-analogue of \cite[Theorem 6.11]{crapo2005unique}.

    \begin{theorem}\label{thm:no_pinchp_iff_irred}
        For any non-uniform $q$-matroid $M$ with ground space $E$, the following are equivalent:
        \begin{enumerate}
            \item $M$ is irreducible with respect to the free product.
            \item The lattice $\mathcal{D}(M)$ contains no non-trivial pinchpoint.
        \end{enumerate}
    \end{theorem}
    \begin{proof}
        By Theorem~\ref{thm:free_sep_free_prod}, $M$ is reducible if and only if $M$ has a non-trivial free separator $A\in\mathcal{L}(E)$, which is equivalent to saying that there exist some $\mathcal{Z}_1,\mathcal{Z}_2\subset\mathcal{Z}(M)$ such that $\mathcal{Z}(M)=\mathcal{Z}_1\cup\mathcal{Z}_2$ and $Z_1\leq A\leq Z_2$ for all $Z_1\in\mathcal{Z}_1$ and $Z_2\in\mathcal{Z}_2$. That is equivalent to saying that 
        $$\sum_{Z\in\mathcal{Z}_1}Z\leq A\leq\bigcap_{Z\in\mathcal{Z}_1}Z,$$
        which is possible if and only if the lattice $\mathcal{D}(M)$ contains a non-trivial pinchpoint when $M$ is not uniform.
    \end{proof}

    \begin{notation}
        We denote by $\mathcal{F}(M)$ the set of free separators of $M$. We denote $[A,B]\cap\mathcal{F}(M)$ by $[A,B]_\mathcal{F}$.
    \end{notation}

    The following result, Lemma~\ref{lem:free_sep_isom}, is the $q$-analogue of \cite[Lemma 6.12]{crapo2005unique}. We offer an alternative method of proof using cyclic flats.

    \begin{lemma}\label{lem:free_sep_isom}
        Let $A,B\in\mathcal{F}(M)$ and $A\leq B$. The map from $[A,B]_\mathcal{F}$ to $\mathcal{F}(M[A,B])$ defined by $X\mapsto X/A$ is a lattice isomorphism.
    \end{lemma}
    \begin{proof}
        It suffices to show $\mathcal{Z}(M[A,B])\backslash\{A,B\}\subseteq\mathcal{Z}(M)\cap[A,B]$ since $A$ and $B$ are clearly in $[A,B]_\mathcal{F}$ and $\mathcal{F}(M[A,B])$, and it is clear that $\mathcal{Z}(M)\cap[A,B]\subseteq\mathcal{Z}(M[A,B])$.
    
        Suppose there exists $X\in\mathcal{Z}(M[A,B])\backslash(\mathcal{Z}(M)\cap[A,B])$. Then either $\bar{X}_1:=\text{cl}(\text{cyc}(X))\notin[A,B]$ or $\bar{X}_2:=\text{cyc}(\text{cl}(X))\notin[A,B]$. Since $A$ and $B$ are free separators, we must have $\bar{X}_i<A$ or $\bar{X}_i>B$ for $i=1$ or 2. Therefore, we have $X<\bar{X}_i$ or $\bar{X}_i<X$. Since $X$ is a cyclic flat in $M[A,B]$, we thus have that $X=A$ or $B$, because all other such cyclic flats $Y$ in $[A,B]$ would give $\text{cl}(Y)=\text{cyc}(Y)=Y$.
        Since all cyclic flats not equal to $A$ or $B$ are in one-to-one correspondence, so must the free separators be. The result follows.
    \end{proof}

    The following lemma is the $q$-analogue of \cite[Lemma 6.13]{crapo2005unique} and follows immediately from Proposition~\ref{prop:D(M)_of_uniform} and the definitions of $\mathcal{D}(M)$ and $\mathcal{F}(M)$.
    \begin{lemma}\label{lem:unif_lat_of_free_seps}
        The $q$-matroid $M$ on the lattice $\mathcal{L}(E)$ is uniform if and only if $\mathcal{F}(M)=\mathcal{L}(E)$.
    \end{lemma}

     The next result, Lemma~\ref{lem:prim_flag_free_seps}, is the $q$-analogue of \cite[Proposition 6.15]{crapo2005unique}. Our proof does not significantly differ from that in \cite{crapo2005unique}, and therefore we omit it. We include Lemmas \ref{lem:unif_lat_of_free_seps} and \ref{lem:prim_flag_free_seps} because their Boolean analogues in \cite{crapo2005unique} are used in the proof of \cite[Theorem 6.16]{crapo2005unique}, whose $q$-analogue is Theorem~\ref{thm:unique_prim_factorisation}.

    \begin{lemma}\label{lem:prim_flag_free_seps}
        If the $q$-matroid $M$ has the primary flag $T_0<\dots< T_k$, then the lattice $\mathcal{F}(M)$ of free separators is equal to the union $\bigcup_{i=1}^k[T_{i-1},T_i]_\mathcal{F}$, where each interval $[T_{i-1},T_i]_\mathcal{F}$ is a sublattice of $\mathcal{L}(E)$ with
        $$[T_{i-1},T_i]_\mathcal{F}=\left\{\begin{array}{cc}
           [T_{i-1},T_i]  &  \textup{if }T_i\textup{ covers }T_{i-1}\textup{ in }\mathcal{D}(M),\\
           \{T_{i-1},T_i\}  & \textup{otherwise,}
        \end{array}\right.$$
        for $i\in \{1,\dots,k\}$.
    \end{lemma}

    Note that the following \emph{primary factorisation} is made possible due to an iterative application of Theorem~\ref{thm:free_sep_free_prod}, as well as the associativity of the free product (Proposition~\ref{prop:assoc_dir_prod}).

    \begin{definition}
        Let $M$ be a $q$-matroid with primary flag $T_0<\dots< T_k$. The minor $M[T_{i-1},T_i]$ is a \textbf{primary factor} of $M$, and the factorisation $M=M[T_0,T_1]\square\dots\square M[T_{k-1},T_k]$ is the \textbf{primary factorisation} of $M$.
    \end{definition}

    The following is the $q$-analogue of \cite[Theorem 6.16]{crapo2005unique}. By changing the characterisation of a uniform matroid $M$ from one that has $\mathcal{F}(M)$ equal to a Boolean lattice, to one such that $\mathcal{F}(M)$ is isomorphic to an interval of the ambient lattice $\mathcal{L}(E)$ (given in Lemma~\ref{lem:unif_lat_of_free_seps}), it is a direct $q$-analogue of \cite[Theorem 6.16]{crapo2005unique}. We will provide a proof for clarity and self-containment. While the overall argument we provide here follows the result of \cite{crapo2005unique}, we remark that different approaches were required here to obtain the preceding results applied in its proof.

    \begin{theorem}\label{thm:unique_prim_factorisation}
        The sequence of primary factors of a $q$-matroid $M$ is the unique sequence $M_1,\dots,M_k$ such that $M=M_1\square\cdots\square M_k$ where each $M_i$ is either uniform or irreducible, and no free product of consecutive elements in the sequence is uniform.
    \end{theorem}
    \begin{proof}
        Let $M=M_1\square\cdots\square M_\ell$, and let 
        $\mathcal{U}=\{U_0,\dots U_\ell\}$ be the corresponding set of free separators such that $U_{i-1} < U_i$ and $M_i=M[U_{i-1},U_i]$ for $i \in \{1,\dots,\ell\}$. Let $\mathcal{T}=\{T_0,\dots,T_k\}$, with $T_{j-1}<T_j$ for $j\in\{1,\dots,k\}$ be the primary flag of $M$. We will show that the statement of the theorem is satisfied if and only if $\mathcal{U}=\mathcal{T}$.

        Suppose first that $\mathcal{U}=\mathcal{T}$. By Lemma~\ref{lem:free_sep_isom} we have $\mathcal{F}(M_i)=\mathcal{F}(M[T_{i-1},T_i])\cong[T_{i-1},T_i]_\mathcal{F}$, for $i\in \{1,\dots,k\}$. By Lemma~\ref{lem:prim_flag_free_seps}, combined with Proposition~\ref{prop:D(M)_of_uniform} and Theorem~\ref{thm:no_pinchp_iff_irred}, we deduce that $M_i$ is either uniform, or irreducible. For $i\in \{1,\dots,k-1\}$ we have $\mathcal{F}(M_i\square M_{i+1})\cong[T_{i-1},T_{i+1}]_\mathcal{F}$, which has a non-trivial pinchpoint at $T_i$. Proposition~\ref{prop:D(M)_of_uniform} then gives us that $M_i\square M_{i+1}$ is not uniform.

        For the converse, suppose first that $\mathcal{T}$ is not a subset of $\mathcal{U}$. Since $\mathcal{U}$ is composed of free separators, there then must exist $T_j\in[U_i,U_{i+1}]_\mathcal{F}\backslash\{U_i,U_{i+1}\}$ for some $i,j$. This means that $M_i$ is neither uniform nor irreducible.

        Finally, suppose that $\mathcal{T}$ is a proper subset of $\mathcal{U}$. Then there exists $U_j\in[T_i,T_{i+1}]_\mathcal{F}\backslash\{T_i,T_{i+1}\}$ for some $i,j$. By Lemma~\ref{lem:free_sep_isom} we must have that $M[T_i,T_{i+1}]$ is uniform. Since $\mathcal{T}\subseteq\mathcal{U}$, we must have that $T_i\leq U_{j-1}$ and $U_{j+1}\leq T_{i+1}$, which means that $M_j\square M_{j+1}$ is a minor of $M[T_i,T_{i+1}]$, which means that it must be uniform.
    \end{proof}

    As in the matroid case, we have thus shown that any $q$-matroid factors uniquely into minors that are either irreducible or maximally uniform (i.e. maximal with respect to inclusion among the minors).

    The following two results are the $q$-analogues of \cite[Theorem 6.17, Theorem 6.18]{crapo2005unique}, which we state here as corollaries (of Theorem~\ref{thm:unique_prim_factorisation}), and require no further proof than what is given in \cite{crapo2005unique}.

    \begin{corollary}
        If $M\cong M_1\square\cdots\square M_k\cong N_1\square\cdots\square N_r$, where each factor is irreducible, then $k=r$ and $M_i=N_i$ for all $i\in \{1,\dots,k\}$.
    \end{corollary}

    \begin{corollary}\label{cor:two-prod_isom}
        Let $M_1,M_2,N_1,N_2$ be $q$-matroids on the spaces $E_1,E_2,E_1',E_2'$.
        Suppose that $M_1\square M_2\cong N_1\square N_2$ and $E_1\cong E_1'$. Then $M_1\cong N_1$ and $M_2\cong N_2$.
    \end{corollary}

    By Corollary~\ref{cor:two-prod_isom} we can derive a $q$-analogue of \cite[Corollary 9]{crapo2005free}, which gives a recursive lower bound for the number of isomorphism classes of matroids on a set with $n$ elements. The proof does not differ significantly from the proof found in \cite{crapo2005free}, so we omit it.

    \begin{corollary}
        Let $\mathcal{M}_{q,n}$ denote the set of isomorphism classes of $q$-matroids on the vector space $\mathbb{F}_q^n$. We have $|\mathcal{M}_{q,n}|\geq |\mathcal{M}_{q,n_1}|\cdot |\mathcal{M}_{q,n_2}|$ if $n=n_1+n_2$.
    \end{corollary}

    To give an example of an irreducible $q$-matroid, we include the following definition, taken from \cite[Example 17]{byrne2022weighteddesigns}, which is the $q$-analogue of the well-known Vámos matroid.

    \begin{definition}\label{def:Vamos}
        Let $e_1,\dots,e_8$ be the canonical basis of $\mathbb{F}_q^8$. Define the set
        $$\mathcal{C}=\{\langle e_1,e_2,e_3,e_4\rangle,\langle e_1,e_2,e_5,e_6\rangle,\langle e_3,e_4,e_5,e_5\rangle,\langle e_3,e_4,e_7,e_8\rangle,\langle e_5,e_6,e_7,e_8\rangle\}.$$
        The \emph{Vámos $q$-matroid} is the $q$-matroid on the lattice $\mathcal{L}(\mathbb{F}_q^8)$ with rank function $r:\mathcal{L}(\mathbb{F}_q^8)\rightarrow\mathbb{Z}$ defined by
        $$r(A)=\left\{\begin{array}{cc}
             \text{dim}(A)& \text{dim}(A)\leq3 \\
             3& A\in\mathcal{C} \\
             4& \text{dim}(A)\geq4\textup{ and }A\notin\mathcal{C}.
        \end{array}\right.$$
    \end{definition}

    \begin{example}\label{ex:vamos_irred}
        It is clear that the lattice of cyclic flats of the Vámos $q$-matroid contains the set $\mathcal{C}$. It is easily observed that $\mathcal{C}$ is an antichain, that $\bigwedge\mathcal{C}=\boldsymbol{0}$, and that $\bigvee\mathcal{C}=\mathbb{F}_q^8$. Therefore, there are no non-trivial pinchpoints in its lattice of cyclic flats, which means that it is irreducible by Theorem~\ref{thm:no_pinchp_iff_irred}.
    \end{example}

    \begin{remark}
        Example~\ref{ex:vamos_irred} is not surprising, because the Vámos matroid has the same lattice of cyclic flats as the Vámos $q$-matroid. Therefore, if the Vámos $q$-matroid was reducible, then the Vámos matroid would also be reducible. This would mean that the Vámos matroid is the free product of two smaller matroids, which must then both be representable. By \cite[Proposition 4.13]{crapo2005unique}, this would imply that the Vámos matroid is representable, which famously is not the case.
    \end{remark}
    
    
    \section{Representability of the free product}\label{sec:representability}
    
In this section we study the representability of the free product of 
$\F_{q^m}$-representable $q$-matroids. Using the geometric point of view introduced in \cite{alfarano2022cyclic} for representable $q$-matroids, we focus on the free product of two uniform $q$-matroids, with particular attention to the case for which both factors have rank one.
It will be convenient for us to set $E_1=\langle e_1,\dots,e_{n_1} \rangle$
and $E_2 =E_1^\perp=\langle e_{n_1+1},\dots,e_{n_1+n_2} \rangle$, where $\{e_j : 1 \leq j \leq n_1+n_2\}$ is the standard basis of $\F_q^{n_1+n_2}$.


We recall a geometric description of representable $q$-matroids; see also \cite{ALFARANO2022105658,alfarano2022cyclic,alfarano2024representability}.

\begin{definition}
An $[n,k]_{q^m/q}$ \textbf{system} $\mS$ is an $n$-dimensional $\F_q$-subspace of $\F_{q^m}^k$, such that
$ \langle \mS \rangle_{\F_{q^m}}=\F_{q^m}^k$.
Two $[n,k]_{q^m/q}$ systems $\mS$ and $\mS'$ are called \textbf{equivalent} if there exists an $\F_{q^m}$-isomorphism $\tau\in\GL(k,\F_{q^m})$ such that $ \tau(\mS) = \mS'$.
If the parameters are not relevant or clear from the context, we simply say that $\mS$ is a $q$-\textbf{system}.
\end{definition}

\begin{definition}\label{def:rhos}
    For an $\F_q$-subspace $V\leq \F_{q^m}^k$, we define the \textbf{$\F_{q^m}$-rank of $V$} to be the integer
    $$r(V)=\dim_{\F_{q^m}}(\langle V \rangle_{\F_{q^m}}).$$
    We write $r_\mS$ to denote the restriction of the map 
    $r:\mL(\F_{q^m}^k) \longrightarrow \N_0$ to $\mL(\mS)$.
\end{definition}

It is not difficult to see that for a $q$-system $\mS$, $r_\mS$ is a rank function and hence  $(\mS,r_{\mS})$ defines the $q$-matroid $(\mS,r_{\mS})$.
The next geometric interpretation of a representable $q$-matroid in terms of $q$-systems was proved in \cite[Theorem 5.6]{alfarano2022cyclic}, whose original statement may look slightly different. However, it is equivalent to the following reformulation.

 \begin{theorem}[{\cite[Theorem 5.6]{alfarano2022cyclic}}]\label{thm:independentrank}
Let $G\leq \F_{q^m}^{k\times n}$ be a full rank matrix, and let $\mS_G$ be the $[n,k]_{q^m/q}$ system associated to it, i.e. the $\F_q$-span of the columns of $G$. Then the $q$-matroid $M[G]$ arising from $G$ is equivalent to the $q$-matroid $(\mS_G,r_{\mS_G})$. 
\end{theorem}

\begin{remark}\label{rem:Fqm-indep}
    Let $M=(\mS,r_{\mS})$ be the representable $q$-matroid arising from the $[n,k]_{q^m/q}$-system~$\mS$. We recall that in \cite{alfarano2022cyclic}, it has been shown that the independent spaces of $M$ are the \textbf{$\F_{q^m}$-independent subspaces} of $\mS$. An $\F_q$-subspace $I$ of $\mS$ is $\F_{q^m}$-independent if $r_\mS(I)=\dim_{\F_q}(I)$. 
In other words, Theorem \ref{thm:independentrank} characterises representable $q$-matroids as those coming from a $q$-system. More precisely, we say that a $q$-matroid $M$ of rank $k$ without loops is \textbf{$\Fm$-representable} if and only if it is equivalent to  a $q$-matroid $(\mS,r_\mS)$, for some $[n,k]_{q^m/q}$ system~$\mS$.
\end{remark}

We turn to the representability of the free product, with the following result.

    \begin{proposition}
        Let $M_1=(E_1,r_1)$ and $M_2=(E_2,r_2)$ be $q$-matroids of rank $k_1$ and $k_2$, respectively. If $M_1\square M_2$ is $\F_{q^m}$-representable, then $M_1$ and $M_2$ are $\F_{q^m}$-representable, and $M_1\square M_2=M[G]$ for some matrix 
        $$G=\begin{pmatrix}
            G_1&X\\
            0&G_2
        \end{pmatrix},$$
        where $G_1$ represents $M_1$, $G_2$ represents $M_2$, and $X \in \F_{q^m}^{k_1 \times n_2}$ is such that 
        $|\{ U \leq \mS_G : \dim_{\F_q}(U) = \dim_{\F_{q^m}}(\langle U\rangle_{\F_{q^m}}) \}|$ is maximal, over all such choices of $X$.
    \end{proposition}
    \begin{proof}
        Suppose that $M_1\square M_2$ is $\F_{q^m}$-representable. Then there exists a matrix $G \in \F_{q^m}^{(k_1+k_2) \times (n_1+n_2)}$ that represents $M_1\square M_2$ and which may be assumed to have the form
        $$G=\begin{pmatrix}
            G_1&X\\
            0&G_2
        \end{pmatrix},$$
        for some $G_i \in \F_{q^m}^{k_i \times n_i}$ and $X \in \F_{q^m}^{k_1 \times n_2}$.
        Clearly, $(M_1\square M_2)|{E_1}$ and $(M_1\square M_2)/E_1$ are represented by $G_1$ and $G_2$, respectively. Since $(M_1\square M_2)|{E_1}\cong M_1$ and $(M_1\square M_2)/E_1\cong M_2$, we have that $M_i$ is represented by $G_i$ for each $i$. The rest follows from Corollary \ref{cor:maximality_free_prod}, since $M_1\square M_2$ is maximal with respect to the weak ordering on ${\mathcal M}_q(M_1,M_2)$.
    \end{proof}

For the rest of this section, we will adopt the following notation.

\begin{notation}\label{Not:q-systems}
Let $\mS_1$ and $\mS_2$ be two $[n_1,k_1]_{q^m/q}$ and $[n_2,k_2]_{q^m/q}$-systems, respectively. Let $\{u_1,\ldots,u_{n_1}\}$ be a basis of $\mS_1$ and $\{w_1,\ldots,w_{n_2}\}$ be a basis of $\mS_2$. Let $x_1,\ldots,x_{n_2}\in\F_{q^m}^{k_1}$ be some fixed vectors and let $X\in\F_{q^m}^{k_1\times n_2}$ be the matrix whose columns are $x_1,\ldots,x_{n_2}$.
For $i=1,\ldots,n_1$, we define $\overline{u}_i:=(u_i,0,\dots,0)\in\F_{q^m}^{k_1+k_2}$ and for every $j=1,\ldots,n_2$, we define $\widetilde{w}_j:=(x_j,w_j)\in\F_{q^m}^{k_1+k_2}$.
We define the following embeddings:
\begin{align}
    \varphi_1: \mS_1 &\hookrightarrow \F_{q^m}^{k_1+k_2}, \:\:\sum_{j=1}^{n_1} a_j u_j \mapsto \sum_{j=1}^{n_1} a_j \overline{u}_j\:\:\forall\: a_j \in \F_q ;\label{eq:phi1}\\
    \varphi_2^X: \mS_2 &\hookrightarrow \F_{q^m}^{k_1+k_2},\:\:\sum_{j=1}^{n_w} b_j w_j \mapsto \sum_{j=1}^{n_w} b_j \tilde{w}_j\:\:\forall\: b_j \in \F_q .\label{eq:phi2}
\end{align}
Moreover, we define:
$$\sigma_X:\mS_2 \to \langle x_1,\ldots,x_{n_2}\rangle_{\Fq}, w_i\mapsto x_i.$$

Let $M_1, M_2$ be two $\F_{q^m}$-representable $q$-matroids with ground spaces $\F_q^{n_1}$ and $\F_{q}^{n_2}$ and let $G_1\in\F_{q^m}^{k_1\times n_1}$, $G_2\in\F_{q^m}^{k_2\times n_2}$ be matrix representations of $M_1$ and $M_2$, respectively. Let $\mS_1,\mS_2$ be the $q$-systems associated to $G_1$ and $G_2$ (i.e. the respective $\F_q$-spans of their columns). Define the isomorphisms $\psi_{G_1}:\F_q^{n_1}\to\mS_1, \; v\mapsto G_1v^\top$ and 
$\psi_{G_2}:\F_q^{n_2}\to\mS_2, \; v\mapsto G_2v^\top$. 
Consider the $q$-system~$\mS$ associated to 
\begin{equation}\label{eq:G}
    G=\begin{pmatrix}
    G_1 & X \\
    0 & G_2
\end{pmatrix}.
\end{equation} 
Note that $\mS$ is an $(n_1+n_2)$-dimensional $\F_q$-subspace of $\F_{q^m}^{k_1+k_2}$
and hence there is a natural isomorphism $\psi_G: \F_{q}^{n_1+n_2}\to \mS$.
\end{notation}

    Clearly, $\dim_{\F_q}(\mS_1)=\dim_{\F_q}(\varphi_1(\mS_1))$ and $\dim_{\F_q}(\mS_2)=\dim_{\F_q}(\varphi_2^X(\mS_1))$.\\ 
    Moreover, $\langle \varphi_1(\mS_1), \varphi_2^X(\mS_2)\rangle_{\F_q}\cong \mS$ for every $X\in\F_{q^m}^{k_1\times n_2}$.
We also make the following observation.
    
\begin{lemma}
    Given Notation \ref{Not:q-systems}, let $A\leq\F_q^{n_1}$ and $B\leq \F_{q}^{n_2}$.
    The space 
    $$\langle \varphi_1(\psi_{G_1}(A)),\varphi_2^X(\psi_{G_2}(B))\rangle_{\Fq}$$
    is an $\F_q$-subspace of $\mS$.
\end{lemma}
\begin{proof}
   Note that a basis of $\mS$ is given by $\{\bar{u}_1,\ldots, \bar{u}_{n_1}, \tilde{w}_1,\ldots,\tilde{w}_{n_2}\}$, as defined in Notation \ref{Not:q-systems}. 
   We have that $\varphi_1(\psi_{G_1}(A))$ is an $\F_q$-subspace of $\langle \bar{u}_1,\ldots, \bar{u}_{n_1}\rangle_{\F_q}$ and $\varphi_2^X(\psi_{G_2}(B))$ is an $\F_q$-subspace of $\langle  \tilde{w}_1,\ldots,\tilde{w}_{n_2}\rangle_{\F_q}$. Hence, the statement follows.
\end{proof}

\subsection{The free product of uniform \emph{q}-matroids}\label{sec:uniforms}
This subsection is inspired by the results of \cite{alfarano2024representability}, wherein the representability of the direct sum of uniform $q$-matroids is investigated from a geometric point of view. We will give a necessary condition for a $q$-system to be the representation of the free product of two uniform $q$-matroids. It turns out that for both the direct sum and the free product of uniform $q$-matroids, the cyclic flats play a key role.
We then restrict our attention to the special case of the free product of a pair of rank one uniform $q$-matroids. In this case, we show that a necessary condition for representability is equivalent to the existence of \emph{clubs} on the $\F_{q^m}$-projective line.

\medskip

We denote by $\PG(k-1,q^m)$ the $(k-1)$-dimensional projective space with underlined vector space $\F_{q^m}^k$. For a vector space $V$, we denote by $\PG(V,\F_{q^m})$, its associated projective space.

\begin{definition}
Let $\mS$ be an $[n,k]_{q^m/q}$ system. For each $\F_q$-subspace $V\leq \F_{q^m}^k$, we define the \textbf{weight} of $V$ in $\mS$ to be the integer
$$ \wt_{\mS}(V):=\dim_{\F_q}(\mS\cap V).$$
The set 
$$L_{\mS}:=\{\langle v\rangle_{\F_{q^m}} : v\in\mS\setminus\{0\}\}$$
is called the \textbf{linear set} of $\mS$ of rank $\dim_{\F_q}(\mS)$ in $\PG(k-1,q^m)$. 
We also define the \textbf{weight} of the projective space $\PG(V,\F_{q^m})$ in $L_\mS$ to be $\wt_{\mS}(V)$, which we denote by $\wt_{L_{\mS}}(\PG(V,\F_{q^m}))$.
\end{definition}

We could say that a $q$-system is the vectorial counterpart of a linear set. For our purposes, we do not need a deep background in finite geometry. However, for a nice treatment of linear sets, we refer the interested reader to \cite{polverino2010linear}.

We now introduce the concept of an \emph{evasive} space, which is a $q$-system with special intersection properties. These objects have been studied in \cite{bartoli2021evasive} as a $q$-analogue of \emph{evasive sets}. 
An evasive subspace is a natural generalization of an $h$-scattered subspace; see \cite{blokhuis2000scattered, csajbok2021generalising}.
Moreover, there are connections to rank-metric codes; see \cite{bartoli2021evasive,marino2023evasive}. 

\begin{definition} \label{def:ahevasive}
Let $h$ be a positive integer. 
Let $\mathcal{A}$ be a collection of $\F_q$-subspaces of $\F_{q^m}^k$, and let $\mS$ be an $[n,k]_{q^m/q}$ system. We say that
$\mS$ is $(\mathcal{A},h)$-\textbf{evasive} if
\[ \wt_{\mS}(A)\leq h \quad \mbox{for all $A \in \mathcal{A}$}. \]
\end{definition}

\begin{definition}
   
    Let $\mS$ be an $[n,k]_{q^m/q}$ system and let $h,r$ be positive integers such that $0\le h \le k$. We denote by $\Lambda_h$ the set of all the $h$-dimensional 
$\F_{q^m}$-subspaces of $\F_{q^m}^k$. We say that $\mS$ is \textbf{$(h,r)$-evasive} if $\mS$ is $(\Lambda_h,r)$-\textbf{evasive}. When $h=r$, we say that $\mS$ is \textbf{$h$-scattered}. Finally, for $h=1$ a $1$-scattered $q$-system will be simply called \textbf{scattered}. 
\end{definition}

Let $k_1,k_2,k,h$ be positive integers such that $k=k_1+k_2$ and $1\leq h \leq k-1$. We denote by $\Lambda_{h,k_1}$ the set of all the $h$-dimensional $\F_{q^m}$-subspaces of $\F_{q^m}^k$ that do not contain $\Fm^{{k_1}}\oplus \zero$.

From Theorem \ref{thm:char_uniform} and Lemma \ref{lem:indep_charact}, we have the following result. 

\begin{corollary}\label{cor:ind_free_uniforms}
    Let $\mU_{k_i,n_i}(q)$, $i=1,2$ be uniform $q$-matroids. Let $M=\mU_{k_1,n_1}(q)\square \mU_{k_2,n_2}(q)$. Then $I\in\mI(M)$ if and only if
    $$\dim(I)\leq k_1+k_2,  \;  \textnormal{ and } \; \dim(I\cap (\F_q^{n_1}\oplus \zero))\leq k_1.$$
\end{corollary}

\begin{notation}
    Let $\mS_1$ be an $[n_1,k_1]_{q^m/q}$-system and let $\mS_2$ be an $[n_2,k_2]_{q^m/q}$-system. Let $X\in\F_{q^m}^{k_1\times n_2}$ and define the $[n_1+n_2,k_1+k_2]_{q^m/q}$-system $\mS_1\square_X\mS_2$ to be $\langle \varphi_1(\mS_1),\varphi_2^X(\mS_2)\rangle_{\F_q}\leq \F_{q^m}^{k_1+k_2}$, where $\varphi_1$ and $\varphi_2^X$ are defined as in \eqref{eq:phi1} and \eqref{eq:phi2}, respectively.
\end{notation}

The following result characterises the independent spaces of an $\Fm$-representation of the free product of $\mU_{k_1,n_1}(q)$ and $\mU_{k_2,n_2}(q)$.

\begin{proposition}\label{prop:geo_prod_unif}
    Let $k_1,k_2,k,n_1,n_2, m$ be positive integers satisfying $1\le k_i< n_i\le m$ for $i =1,2$ and $k=k_1+k_2$. Let $\mS_i$ be an $\Fm$-representation of $\mU_{k_i,n_i}(q)$ for each $i\in \{1,2\}$. The following statements are equivalent.
    \begin{enumerate}
        \item[(i)] $(\mS_1\square_X\mS_2,r_{\mS_1\square_X\mS_2})$ is an $\Fm$-representation of $\mU_{k_1,n_1}(q)\square\mU_{k_2,n_2}(q)$.
    \item[(ii)] For any $\Fq$-subspace $I\subseteq \mS_1\square_X \mS_2$ we have that 
\begin{equation*}
        r_{\mS_1\square_X \mS_2}(I)=\dim_{\F_q} (I) \Longleftrightarrow 
            \begin{cases} \dim_{\Fq}(I)\le k \text{ and}\\
                \dim_{\F_q} (I\cap (\mS_1\oplus \zero)) \leq k_1.
            \end{cases} 
    \end{equation*}
    \end{enumerate}
\end{proposition}
\begin{proof}
    Let $M_1=(\mS_1\square_X \mS_2,r_{\mS_1\square_X \mS_2})$ and let $M_2=\mU_{k_1,n_1}(q)\square\mU_{k_2,n_2}(q)$. Then, $M_1\cong M_2$ if and only if there exists an invertible $\Fq$-linear map $\psi:\Fq^{n_1+n_2}\longrightarrow\mS_1\square_X\mS_2$ such that $\mI(M_1)=\mI(\psi(M_2))$. Let $G$ be any generator matrix associated to $\mS_1\square_X\mS_2$ of the form
        $$G=\begin{pmatrix}
            G_1 & X \\ 0 & G_2
        \end{pmatrix},$$
        where $G_1\in \F_{q^m}^{k_1\times n_1}$, $G_2\in \F_{q^m}^{k_2\times n_2}$ are generator matrices associated with $\mS_1$ and $\mS_2$, respectively. The statement now follows from the characterization of independent spaces of $M_1$ given in Remark \ref{rem:Fqm-indep}, and the characterization of independent spaces of $M_2$ derived in Corollary \ref{cor:ind_free_uniforms}.
\end{proof}

The following result illustrates a necessary condition for $\mS_1\square_X\mS_2$ to be an $\F_{q^m}$-representation of $\mU_{k_1,n_1}(q)\square\mU_{k_2,n_2}(q)$.

\begin{theorem}\label{thm:repr_evasivity}
    Let $k_1,k_2,k,n_1,n_2, m$ be positive integers, with $1\le k_i< n_i\le m$ for $i \in \{1,2\}$ and $k=k_1+k_2$. Let $\mS_i$ be an $\Fm$-representation of $\mU_{k_i,n_i}(q)$ for each $i\in \{1,2\}$. If $\mS_1\square_X\mS_2$ is an $\Fm$-representation of $\mU_{k_1,n_1}(q)\square\mU_{k_2,n_2}(q)$, then $\mS_1\square_X\mS_2$ is $(\Lambda_{k-1,k_1},k-1)$-evasive.
\end{theorem}
\begin{proof}
    Assume, towards a contradiction, that $\mS_1\square_X\mS_2$ is not $(\Lambda_{k-1,k_1},k-1)$-evasive. Then, there exists an $\F_{q^m}$-hyperplane $H\subseteq\F_{q^m}^{k_1+k_2}$, such that $H$ does not contain $\F_{q^m}^k\oplus \zero$ and $\dim_{\F_q}(H\cap(\mS_1\square_X\mS_2))\geq k_1+k_2$. Set $V=H\cap(\mS_1\square_X\mS_2)$ and consider $I\leq V$, such that $\dim_{\F_q}(I)=k_1+k_2$.  Since $\langle I\rangle_{\F_{q^m}}\leq H$, we have that 
    $$r_{\mS_1\square_X\mS_2}(I)=\dim_{\F_q^m}\langle I\rangle_{\F_{q^m}}<k_1+k_2.$$
    Hence, $I$ is not independent and so by Proposition \ref{prop:geo_prod_unif}, $\mS_1\square_X\mS_2$ cannot be a representation of $\mU_{k_1,n_1}(q)\square\mU_{k_2,n_2}(q)$.
\end{proof}

Theorem \ref{thm:repr_evasivity} shows that in order to find a representation of 
$\mU_{k_1,n_1}(q)\square \mU_{k_2,n_2}(q)$, for which each components $\mU_{k_i,n_i}(q)$ has an $\Fm$-representations $\mS_i$, we must find a $(k_1\times n_2)$ matrix $X$ such that $\mS_1\square_X\mS_2$ is $(\Lambda_{k_1+k_2-1,k_1},k_1+k_2-1)$-evasive.

For the remainder of this section, we set $k_1=k_2=1$. 
In this case, we will observe that the linear set associated to a $(\Lambda_{1,1},1)$-evasive $(n_1+n_2)$-dimensional $\F_q$-subspace of $\F_{q^m}^{2}$ is an $n_1$-\emph{club} of rank $(n_1+n_2)$ in $\PG(1,q^m)$. These are well-studied objects in finite geometry, which were introduced in 2006 by Fancsali and Sziklai in the seminal paper \cite{fancsali2006maximal}. 

\begin{definition}
    An \textbf{$i$-club of rank} $n$ in $\PG(1,q^m)$ is an $\F_q$-linear set $L_\mS$ of rank $k$ for which all but one of its elements have weight one, while exactly one element has weight $i$. 
\end{definition}

It is known that in order for a club to exist, we must have $n\leq m$. In the case that the rank is maximal, i.e. $n=m$, $L_{\mS}$ is simply called an \textbf{$i$-club}. In the literature, $i$-clubs are the most studied among clubs. The interest in clubs was renewed when De Boeck and Van de Voorde in \cite{de2016linear} characterised the translation KM-arcs (\cite{korchmaros1990q}) exactly as those
that can be described by $i$-clubs in even characteristic. A first algebraic construction was already given in \cite{korchmaros1990q} and the corresponding geometrical construction can be found in \cite{gacs2003q}. Moreover, $i$-clubs also define linear blocking sets of R\'edei type and they define Hamming metric codes with few weights; see \cite{napolitano2023two}. The algebraic description of $i$-clubs has been recently
investigated in \cite{bartoli2022r}, under the name of $1$-fat polynomials; see also \cite{polverino2024fat}. Constructions of $(n-1)$- and $(n-2)$-clubs in $\PG(1,q^n)$, as well as $t(\ell-1)$- and $t(\ell-1)+1$-clubs in $\PG(1,q^{rt})$ are known; see \cite{napolitano2022clubs} and the references therein. Finally, we observe that in \cite{de2022weight} it has been shown that $2$-clubs of rank $5$ in $\PG(1,q^5)$ do not exist.

The following result is a specialization of Theorem \ref{thm:repr_evasivity} when $k_1=k_2=1$.

\begin{theorem}\label{thm:evasive-clubs}
     Let $n_1,n_2, m$ be positive integers, with $1\leq n_i\le m$ for $i \in \{1,2\}$ and $n=n_1+n_2$. Let $\mS_i$ be an $\Fm$-representation of $\mU_{1,n_i}(q)$ for each $i\in \{1,2\}$ and let $\mS:=\mS_1\square_X\mS_2$. The following are equivalent.
     \begin{enumerate}
         \item[(i)] $(\mS,r_{\mS})$ is an $\Fm$-representation of $\mU_{1,n_1}(q)\square\mU_{1,n_2}(q)$. 
         \item[(ii)] $L_{\mS}$ is an $n_1$-club of rank $n$ in $\PG(1,q^m)$.
     \end{enumerate}
\end{theorem}
\begin{proof}
    \underline{(i) $\Longrightarrow$ (ii)}: By Theorem \ref{thm:repr_evasivity}, $\mS$ is $(\Lambda_{1,1},1)$-evasive. This means that every point in $L_{\mS}$ has weight equal to $1$, except the point $\langle (1,0)\rangle_{\F_{q^m}}$, which has weight $n_1$. Hence, $L_{\mS}$ is an $n_1$-club of rank $n$ in $\PG(1,q^m)$.

    \underline{(ii) $\Longrightarrow$ (i)}: Let $P=\langle (1,0)\rangle_{\F_{q^m}}\in\PG(1,q^m)$. Without loss of generality, we may assume that $P$ is the element of $L_\mS$ of weight equal to $n_1$. By the definition of a club, all the other points $\langle (x,y)\rangle_{\F_{q^m}}$ in $L_\mS$ have multiplicity $1$, hence $\mS$ is $(\Lambda_{1,1},1)$-evasive. 
    Let $P_1=\langle (x_1,y_1)\rangle_{\F_{q^m}},P_2=\langle (x_2,y_2)\rangle_{\F_{q^m}}\in L_\mS$, with $y_1, y_2\ne 0$ and $y_1\ne y_2$. If $y_2 = \lambda y_1$, with $\lambda \in\F_q$, then we can assume without loss of generality that $P_1=\langle (0,1)\rangle_{\F_{q^m}}$. Then the point $\wt_\mS(P_2)\geq 2$, since $x_2(P+\lambda P_1)=P_2$. This implies that for every pair of points $\langle (x_1,y_1)\rangle_{\F_{q^m}}$, and $\langle (x_2,y_2)\rangle_{\F_{q^m}}$ that are different from $P$, it must be the case that $y_1,y_2$ are linearly independent over $\Fq$. Hence, $(\mS,r_\mS)$ is a representation of $\mU_{1,n_1}(q)\square\mU_{1,n_2}(q)$.
\end{proof}

From Theorem \ref{thm:evasive-clubs}, we see that if $\mS_1\square_X\mS_2$ represents $\mU_{1,n_1}(q)\square\mU_{1,n_2}(q)$, then it has to be an $n_1$-club of rank $n_1+n_2$ in $\PG(1,q^n)$.  We remark that this condition is not easy to handle and, in particular, depends on the choice of representations $\mS_1$ and $\mS_2$ as the next example illustrates.

\begin{example}\label{ex:dependence_from_Gi}
    Consider $\mU_{1,2}(q)$, which is known to be $\F_{q^m}$-representable for $m\geq 2$; see \cite[Example 2.4]{gluesing2022representability}. Moreover, for every $m\geq 2$, any matrix $\begin{pmatrix}
        1 & \alpha
    \end{pmatrix}$, with $\alpha\in\F_{q^m}\setminus\F_q$, is a representation of $\mU_{1,2}(q)$. Suppose that $\mU_{1,2}(q)\square\mU_{1,2}(q)$ is represented over $\F_{q^m}$ by the matrix 
    $$G=\begin{pmatrix}
        G_1 & X \\
        0 & G_2
    \end{pmatrix},$$
    where $G_1$ and $G_2$ represent $\mU_{1,2}(q)$ and $X=\begin{pmatrix}
        x_1 & x_2
    \end{pmatrix}$. As we observed before, the linear set arising from the $q$-system associated with $G$ must be a $2$-club of rank $4$ in $\PG(1,q^m)$ and hence $m\geq 4$. Moreover, we may assume that $x_1=0$.
    Let $q=2$, $m=4$ and $\F_{2^4}=\F_2(\alpha)$ with $\alpha^4=\alpha+1$. Choose $G_1=\begin{pmatrix}
        1 & \alpha
    \end{pmatrix}$,
    $G_2 = \begin{pmatrix}
        1 & \alpha^4 \\
    \end{pmatrix}$ and $x_2=\alpha^{11}$.  We obtain that 
    $$G=\begin{pmatrix}
        1 & \alpha & 0 & \alpha^{11} \\
        0 & 0 & 1 & \alpha^4 
    \end{pmatrix}$$
    is an $\F_{2^4}$-representation of $\mU_{1,2}(2)\square\mU_{1,2}(2)$. 
    In order to check this, we considered the $q$-matroid arising from $G$ and with the aid of \textsc{magma} we found that its cyclic flats are exactly $ \boldsymbol{0} , \F_2^2\oplus  \boldsymbol{0} , \F_2^4$, whose ranks are $0,1,2$, respectively. Hence, they are the cyclic flats of the free product of two uniform $q$-matroids, by Theorem \ref{thm:char_uniform}.    
    In the same way can observe that it is not possible to find any $x_2$ such that the matrix 
    $$G=\begin{pmatrix}
        1 & \alpha & 0 & x_2 \\
        0 & 0 & 1 & \alpha^2
    \end{pmatrix}\in\F_{2^4}^{2\times 4}$$
    represents $\mU_{1,2}(2)\square\mU_{1,2}(2)$ over $\F_{2^4}$. 
\end{example}

We conclude this section by providing an example of a $2$-club of rank $5$ in $\PG(1,2^7)$. To the best of our knowledge, such clubs have not been constructed previously. By Theorem \ref{thm:evasive-clubs}, constructing a club $L_\mS\subseteq \PG(1,2^7)$ is equivalent to finding an $\F_{2^7}$-representation $(\mS,r_{\mS})$ of $\mU_{1,2}(2)\square\mU_{1,3}(2)$. 

Let $\F_{2^7}=\F_2(\alpha)$, where $\alpha^7+\alpha+1=0$. Let $\mU_{1,2}(2)$ be represented over $\F_{2^7}$ by $\begin{pmatrix}
    1 & \alpha
\end{pmatrix}$ and
let $\mU_{1,3}(2)$ be represented over $\F_{2^7}$ by $\begin{pmatrix}
    1 & \alpha^2 & \alpha^8
\end{pmatrix}$. 
Let $\mS$ be the $q$-system associated with 
$$G=\begin{pmatrix}
    1 & \alpha & 0 & \alpha^{36} & \alpha^{24} \\
    0 & 0 & 1 & \alpha^2 & \alpha^8
\end{pmatrix}.$$
With the aid of \textsc{magma}, we find that the cyclic flats of $M[G]$ are exactly $ \boldsymbol{0} , \F_2^2\oplus  \boldsymbol{0} , \F_2^5$. By Theorem \ref{thm:char_uniform}, we have that $\mS$ is an $\F_{2^7}$-representation of $\mU_{1,2}(2)\square\mU_{1,3}(2)$ and hence $L_\mS$ is a $2$-club of rank $5$ in $\PG(1,2^7)$.


\section{Open Problems}\label{sec:conclusions}
We have initiated the study of the free product and on the representibility of the free product. 
Our study gives rise to several research questions. We list a few of them.

\begin{enumerate}
    \item We showed in Example~\ref{eg:dir_sum_not_minimal} that $M_1\oplus M_2$ is not in general minimal in $\mathcal{M}_q(M_1,M_2)$. In fact, if $2\leq\min\{|\mathcal{Z}(M_1)|,|\mathcal{Z}(M_2)|\}$, then it is not minimal. This marked difference between matroids and $q$-matroids, raises the question of deriving a formula (if possible) for the number of isomorphism classes of $q$-matroids on $\mathbb{F}_q^n$ with variables the number of isomorphism classes of matroids on $n$ elements, as well as $q$.
    \item We showed in Example \ref{ex:dependence_from_Gi} that the existence of $X$ depends on the choice of the representations of $M_1$ and $M_2$ and on the field size $q^m$. In particular, it would be interesting to establish the properties that a matrix $X$ should satisfy in order to provide conditions for the representability of the free product.
    \item In Theorem \ref{thm:repr_evasivity} we showed that in order to provide a representation of the free product of two uniform $q$-matroids of ranks $k_1$ and $k_2$, it is necessary to find a $q$-system which is also $(\Lambda_{k_1+k_2-1,k_1},k_1+k_2-1)$ evasive. It is an open problem to know if this condition is also sufficient.
    \item Most of the literature on $i$-clubs involves clubs of maximum rank, which do not exist for all parameters. Most relevant to our results, is to know the smallest field $q^m$, for which it is possible to find an $i$-club on $\PG(1,q^m)$. This will provide the smallest field over which the free product of rank one uniform $q$-matroids is representable.
\end{enumerate}

\bigskip
\bigskip

\section*{Acknowledgments}
The authors are thankful to Alessandro Neri and Paolo Santonastaso for fruitful discussion.
G.~N.~A. was supported by the Swiss National Foundation through grant no. 210966. This publication has emanated from research conducted with the financial support of Science Foundation Ireland (18/CRT/6049) and the European Union MSCA Doctoral Networks, (HORIZON-MSCA-2021-DN-01, Project 101072316).
\bigskip
\bigskip
\bigskip 
\bibliographystyle{abbrv}
\bibliography{references}

\end{document}